\documentclass[10pt]{amsart}

\usepackage{mathrsfs,amssymb,txfonts,amsmath}

\input xy
\xyoption{all}

\newcommand{\RR}{\mathbb{R}}
\newcommand{\OO}{\mathscr{O}}

\newcommand{\ZZ}{\mathbb{Z}}

\newcommand{\MM}{\mathscr{M}}

\newcommand{\Cfield}{\mathbb{C}}

\newcommand{\mm}{\mathfrak{m}}

\newcommand{\ep}{\varepsilon}
\newcommand{\Spec}{\textnormal{Spec}\,}

\newcommand{\image}{\textnormal{im}\,}
\newcommand{\kernel}{\textnormal{ker}\,}
\newcommand{\cokernel}{\textnormal{coker}\,}
\newcommand{\degree}{\textnormal{deg}\,}

\newcommand{\HHom}{\mathscr{H}om} %for sheaf
\newcommand{\Hom}{\textnormal{Hom}}
\newcommand{\dimension}{\textnormal{dim}\,}

\newcommand{\rank}{\textnormal{rank}\,}

\newcommand{\Ext}{\textnormal{Ext}}

\newcommand{\Aut}{\textnormal{Aut}}

\newcommand{\al}{\alpha}
\newcommand{\cone}{\textnormal{cone}}

\newcommand{\Coh}{\textnormal{Coh}}

\newcommand{\Ap}{\mathcal{A}^p}

\newcommand{\arinj}{\ar@{^{(}->}}
\newcommand{\arsurj}{\ar@{->>}}
\newcommand{\areq}{\ar@{=}}
\newcommand{\iotam}{{\iota_m}}
\newcommand{\Lo}{\overset{L}{\otimes}}
\newcommand{\HH}{\mathcal{H}}

\newtheorem{theorem}{Theorem}[section]
\newtheorem{lemma}[theorem]{Lemma}
\newtheorem{coro}[theorem]{Corollary}
\newtheorem{pro}[theorem]{Proposition}

\newtheorem{definition}[theorem]{Definition}

\begin{document}
    \title{Moduli of PT-semistable objects II}

\author{Jason Lo}
\address{Department of Mathematics \\ Building 380\\Stanford University \\ Stanford CA 94305}
\curraddr{Department of Mathematics \\  202 Mathematical Sciences Building \\ University of Missouri \\ Columbia MO 65211}
\email{locc@missouri.edu}
\thanks{This paper and its predecessor \cite{Lo1} grew out of my doctoral thesis  at Stanford University.  I would like to thank my thesis advisor, Jun Li, for his patience and constant encouragement while I worked on this project.  I  also thank Arend Bayer for generously providing comments and suggestions as this work took shape.  For the  useful conversations, I thank Brian Conrad, Young-Hoon Kiem, Ravi Vakil and Ziyu Zhang.  For graciously answering my questions at various points in time, I thank Dan Edidin, Max Lieblich, Amnon Neeman, Alexander Polishchuk and Yukinobu Toda.  Lastly, I want to thank the referee for invaluable comments that helped improve the exposition.}


\begin{abstract}
        We generalise the techniques of semistable reduction for flat families of sheaves to the setting of the derived category $D^b(X)$ of coherent sheaves on a smooth projective three-fold $X$.  Then we construct the moduli of PT-semistable objects in $D^b(X)$ as an Artin stack of finite type that is universally closed.  In the absence of strictly semistable objects, we construct the moduli as a proper algebraic space of finite type.
\end{abstract}

\maketitle
%\tableofcontents

\section{Introduction}

In this paper, we continue our study of PT-semistable objects and the construction of their moduli spaces, using the results established in \cite{Lo1}.  In \cite{Lo1}, boundedness of the moduli of PT-semistable objects was established.  We also showed that the stack of objects in the heart used for PT-stability is universally closed, and proved a series of technical lemmas that will now be applied.

Using semistable reduction in the derived category, we now show that PT-semistability is an open property for a flat family of complexes.  This  enables us to construct Artin stacks of PT-semistable objects that are of finite type and universally closed.  When there are no strictly semistable objects, we construct proper algebraic spaces of finite type parametrising PT-stable objects.  Overall, we not only have moduli stacks of objects of any Chern classes in the derived category on three-folds, but also offer a perspective of higher-rank analogues of stable pairs studied in \cite{PT} (see also \cite[Proposition 6.1.1]{BayerPBSC}).

In precise terms, our main theorem is:

\begin{theorem}\label{theorem-main}
Let $(X,H)$ be a polarised smooth projective three-fold over $k$.
\begin{enumerate}
\item The PT-semistable objects on $X$ of any fixed Chern character form a universally closed Artin stack of finite type.  The stack is a substack of Lieblich's Artin stack of universally gluable complexes - see \cite{Lieblich}.
\item When there are no strictly semistable objects, the PT-semistable objects on $X$ of any fixed Chern character form an algebraic space of finite type that satisfies the valuative criterion for properness for an arbitrary discrete valuation ring.  The algebraic space is a subfunctor of Inaba's algebraic space of simple complexes - see \cite{Inaba}.
    \end{enumerate}
\end{theorem}

Although many of the arguments here are written down only for a particular polynomial stability, namely PT-stability on three-folds, the same proofs apply to Gieseker stability for sheaves if one replaces our t-structure by the standard t-structure.  The techniques in this paper should also work for a wider class of stability conditions, and on higher-dimensional varieties.

\subsection{Related Work}

Semistable reduction for sheaves is originally due to Langton \cite{Langton}, while polynomial stabilities were first defined by Bayer \cite{BayerPBSC}.  In the case of smooth projective three-folds $X$, no examples of Bridgeland stability conditions on $D^b(X)$ have been constructed for an arbitrary $X$.  As approximations of Bridgeland stability conditions on three-folds, Bayer \cite{BayerPBSC} and Toda \cite{TodaLSOp} independently came up with the notions of \textit{polynomial stability} and \textit{limit stability}, respectively.

In Bayer's paper, he introduced a class of polynomial stability conditions on normal projective varieties, which includes Toda's limit stability (in fact, Toda's stability acts as a wall in the wall-crossing in \cite{BayerPBSC}) as well as Gieseker stability for sheaves.  One of the main results in Bayer \cite[Proposition 6.1.1]{BayerPBSC} states that, for objects in the heart $\Ap = \langle \Coh_{\leq 1}(X), \Coh_{\geq 2}(X)[1]\rangle$ with $ch=(-1,0,\beta,n)$ and trivial determinant, the stable objects with respect to a particular polynomial stability function are precisely the stable pairs described in Pandharipande and Thomas's work \cite{PT}; for this reason, this particular stability function is called the \textit{PT-stability function} in \cite{BayerPBSC}.  The moduli space of such stable pairs has been constructed by Le Potier \cite{LP} using geometric invariant theory (GIT) and shown to be projective.   In general, however, it is not  clear how to apply GIT to objects in the derived category, because very different-looking complexes can be isomorphic in the derived category.

In Toda's paper \cite{TodaLSOp}, he showed that the moduli space of limit-stable objects in $\Ap$ of $ch=(-1,0,\beta,n)$ and trivial determinant on a  Calabi-Yau three-fold is a separated algebraic space of finite type \cite[Theorem 3.20]{TodaLSOp}.

There were earlier examples of moduli spaces of  objects in the derived category that satisfy the valuative criterion for properness.  For example, in Abramovich and Polishchuk's work \cite{AP}, they showed that the valuative criteria for separatedness and properness for Bridgeland-stable objects hold, under the assumption that the heart of t-structure in the stability condition is \textit{Noetherian}, which $\Ap$ is not.  (The techniques in \cite{AP} were generalised in \cite{Polishchuk}.)  On the other hand, Arcara-Bertram-Lieblich constructed projective moduli spaces of rank-zero Bridgeland-stable derived objects on surfaces with trivial canonical bundle \cite{BSMSKTS}.  We note that the idea of elementary modifications for objects in the derived category  has already appeared in \cite{BSMSKTS}.

In the case of K3 and abelian surfaces $X$, Toda showed in \cite{TodaK3} that for stabilities $\sigma$ lying in a particular connected component of the space $\text{Stab}(X)$ of Bridgeland stability conditions, the moduli of semistable objects with respect to $\sigma$ of any given numerical type and phrase is an Artin stack of finite-type.

\subsection{Notation}

The notations in this paper are the same as those in \cite{Lo1}.  For convenience, we provide a summary of the notations we use.  More details can be found in \cite{Lo1}.

Throughout this paper, $k$ will be an algebraically closed field of characteristic 0.  And $R$ will denote a discrete valuation ring (DVR), not necessarily complete, with uniformiser $\pi$ and field of fractions $K$.  Unless specified, $X$ will always denote a smooth projective three-fold over $k$.

We will write $X_R := X \otimes_k R$, and $X_K := X \otimes_R K$.  For any integer $m \geq 1$, let $X_m := X \otimes_k R/\pi^m$, and let
$$ \iota_m : X_m \hookrightarrow X_R$$
denote the closed immersion.  We will often write $\iota$ for $\iota_1$, and $X_k$ for the central fibre of $X_R$.  For integers $1 \leq m' < m$, let
$$\iota_{m,m'} :  X_{m'} \hookrightarrow X_m$$ denote the closed immersion.  We also write
$$ j : X_K \hookrightarrow X_R$$
for the open immersion.

Note that the pushforward functor $\iota_\ast : \Coh (X_k) \to \Coh (X_R)$ is exact, while the pullback $\iota^\ast : \Coh (X_R) \to \Coh (X)$ is  right-exact.  Similarly for the pushforward ${\iota_{m,m'}}_\ast$.  On the other hand, $j^\ast : \Coh (X_R) \to \Coh (X_K)$ is exact.

For a Noetherian scheme $Y$, we will always write $\mbox{Kom} (Y)$ for the category of chain complexes of coherent sheaves on $Y$, $D^b(Y)$ for the bounded derived category of coherent sheaves, and $D(Y)$ for the unbounded derived category of coherent sheaves.  If $Y$ is of dimension $n$, then for any integer $0 \leq d \leq n$, we define
\begin{align*}
  \Coh_{\leq d}(Y) &= \{ E \in \Coh (Y) : \dimension \text{Supp}\, E \leq d \} \\
  \Coh_{\geq d+1}(Y) &= \{ E \in \Coh (Y) : \Hom_{\Coh (Y)} (F,E)=0 \text{ for all } F \in \Coh_{\leq d}(Y) \}.
\end{align*}
Then $(\Coh_{\leq d}(Y),\Coh_{\geq d+1}(Y))$ is a torsion pair in the abelian category $\Coh (Y)$.  For $0 \leq d' < d$,  we  form the quotient category $\Coh_{d,d'}(Y) := \Coh_{\leq d}(Y)/\Coh_{\leq d'-1} (Y)$, which is an abelian category.  For a coherent sheaf $F$ on $Y$, we write $p(F)$ for its reduced Hilbert polynomial, and if $F \in \Coh_{\leq d}(Y)$, we write $p_{d,d'}(F)$ for its reduced Hilbert polynomial as an element of $\Coh_{d,d'}(Y)$ (see \cite[Section 1.6]{HL}).

For $m \geq 1$, tilting with respect to the torsion pair $( \Coh_{\leq 1}(X_m), \Coh_{\geq 2}(X_m))$ in $\Coh (X_m)$ gives a t-structure with heart
 \begin{align*}
 \Ap_m &:= \Ap (X_m) \\
  &:= \langle \Coh_{\leq 1}(X_m), \Coh_{\geq 2}(X_m)[1]\rangle \\
&= \{ E \in D^b(X) : H^0(E) \in \Coh_{\leq 1}(X), H^{-1}(E) \in \Coh_{\geq 2}(X), \\
&\hspace{5cm} H^i(E)=0 \text{ for all } i\neq 0, -1\}.
 \end{align*}
 on $D^b(X_m)$ (and in fact, on $D(X_m)$ as well - see \cite[Proposition 5.1]{Lo1}).  The truncation functors associated to this t-structure will be denoted by $\tau^{\leq 0}_{\Ap_m}, \tau^{\geq 0}_{\Ap_m}$, and the cohomology functors denoted by $\HH^i_{\Ap_m}$.  We will drop the subscripts when the context is clear.  On any Noetherian scheme $Y$, the cohomology functors with respect to the standard t-structure on $D(Y)$ will always be denoted by $H^i$.  On $X_K$, let $\Ap_K (X_K)$ or $\Ap_K$ denote the heart $\langle \Coh_{\leq 1}(X_K), \Coh_{\geq 2}(X_K)[1]\rangle$.

We will use $D^{\leq 0}_{\Ap_m}, D^{\geq 0}_{\Ap_m}$ to denote the full subcategories of $D(X_m)$
\begin{align*}
  D^{\leq 0}_{\Ap_m} &= \{ E \in D(X_m) : \HH^i(E)=0 \text{ for all } i >0 \} \\
  &= \{ E \in D(X_m) : H^i(E)=0 \text{ for all } i >1, H^0(E) \in \Coh_{\leq 1}(X_m) \}, \\
  D^{\geq 0}_{\Ap_m} &= \{ E \in D(X_m) : \HH^i(E)=0 \text{ for all } i<0 \} \\
  &= \{ E \in D(X_m) : H^i(E)=0 \text{ for all }i<-2, H^{-1}(E) \in \Coh_{\geq 2}(X_m) \}.
\end{align*}

In summary, we have the following maps between the various schemes:
\begin{equation*}
  \iota_{m,m'} : X_{m'} \hookrightarrow X_m, \quad  \iota_m : X_m \to X_R,\quad  j : X_K \hookrightarrow X_R
\end{equation*}
and associated pullback and pushforward functors
\begin{align*}
  \iota_{m,m'}^\ast &: D(X_m) \to D(X_{m'}), \quad {\iota_{m,m'}}_\ast : D(X_{m'}) \to D(X_m) \\
  \iota_m^\ast &: D(X_R) \to D(X_m), \quad \iotam_\ast : D(X_m) \to D(X_R) \\
  j^\ast &: D(X_R) \to D(X_K).
\end{align*}
Since ${\iota_{m,m'}}_\ast$ is exact, it takes $\Ap_{m'}$ into $\Ap_m$.

Since $\iota_m$ is a closed immersion, it is a projective morphism.  Hence we have the adjoint pair $L\iota_m^\ast \dashv \iotam_\ast$, i.e.\ $L\iota_m^\ast$ is the left adjoint, and $\iotam_\ast$ the right adjoint \cite[p.83]{FMTAG}.  Similarly, we have the adjoint pair $ L\iota_{m,m'}^\ast \dashv {\iota_{m,m'}}_\ast$ for any $1\leq m' < m$.

Consistent with the definitions introduced in \cite{AP} and \cite{BSMSKTS}, we will use the following notion of flatness for derived objects:

\begin{definition}
Let $S$ be a Noetherian scheme over $k$, and $X$ a smooth projective three-fold over $k$.  We say an object $E \in D^b(X \times S)$ is a flat family of objects in $\Ap$ over $S$ if, for all closed points $s \in S$, we have
\[
E|_s := L\iota_s^\ast E \in \Ap (X|_s) = \langle \Coh_{\leq 1}(X|_s), \Coh_{\geq 2}(X|_s)[1]\rangle
\]
where $\iota_s : X|_{s} \hookrightarrow X \times S$ is the closed immersion of the fibre over $s$.
\end{definition}

Given a complex $E^\bullet \in D^b(X)$, we say $E^\bullet$ is of dimension $d$ if the dimension of the support of $E^\bullet$, defined to be the union of the supports of the various cohomology $H^i(E^\bullet)$, is $d$.

\subsection{Stability Conditions}

Let $(X,H)$ be a smooth projective variety of dimension $n$ with polarisation $H$.  For any coherent sheaf $F$ on $X$, we  define its degree (with respect to $H$) as $\degree (F) = \int_X c_1(F)\cdot c_1(H)^{n-1}$, and its slope as $\mu (F) = \frac{\degree (F)}{\rank (F)}$.

Polynomial stability was defined on $D^b(X)$ by Bayer for any normal projective variety $X$ \cite[Theorem 3.2.2]{BayerPBSC}. The particular class of polynomial stability conditions we will concern ourselves with for the rest of the paper consists of the following data, where $X$ is a smooth projective three-fold:
\begin{enumerate}
\item the heart $\Ap = \langle \Coh_{\leq 1}(X), \Coh_{\geq 2}(X)[1]\rangle$, and
\item a group homomorphism (the central charge) $Z : K(X) \to \Cfield [m]$ of the form
$$ Z(E)(m) = \sum_{d=0}^3 \int_X  \rho_d H^d  \cdot ch(E) \cdot U \cdot m^d$$
where
\begin{enumerate}
\item the $\rho_d \in \Cfield$ are nonzero and satisfy $\rho_0, \rho_1 \in \mathbb{H}$, $\rho_2, \rho_3 \in -\mathbb{H}$, and $\phi (-\rho_2) > \phi (\rho_0) > \phi (-\rho_3) > \phi (\rho_1)$ (see Figure \ref{figure-PTstab} below),
\item $H \in \text{Amp}(X)_\RR$ is an ample class, and
\item $U =1+U_1 + U_2 + U_3\in A^\ast (X)_\RR$ where $U_i \in A^i(X)$.
\end{enumerate}
\end{enumerate}

\begin{figure*}[h]
\centering
\setlength{\unitlength}{1mm}
\begin{picture}(50,40)
\multiput(0,15)(1,0){50}{\line(1,0){0.5}}
\multiput(25,0)(0,1){40}{\line(0,1){0.5}}
\put(25,15){\vector(-4,1){15}}
\put(2.5,18.6){$-\rho_2$}
\put(25,15){\vector(-1,1){13}}
\put(7.5,28){$\rho_0$}
\put(25,15){\vector(1,2){10}}
\put(30,36){$-\rho_3$}
\put(25,15){\vector(3,1){14}}
\put(39.5,20){$\rho_1$}
\end{picture}
\caption{Configuration of the $\rho_i$ for PT-stability conditions}
\label{figure-PTstab}
\end{figure*}
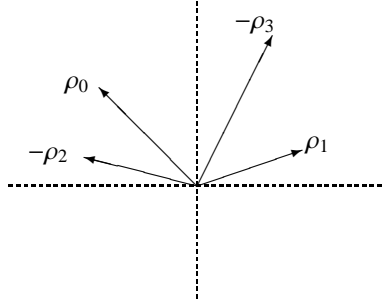

The configuration of the $\rho_i$ is compatible with the heart $\Ap$, in the sense that for every nonzero $E \in \Ap$, we have $Z(E)(m) \in \mathbb{H}$ for $m \gg 0$.  So there is a unqiuely determined function $\phi (E)(m)$ (strictly speaking, a uniquely determined function germ) such that
\[
 Z(E)(m) \in \mathbb{R}_{>0} e^{i \pi \phi (E)(m)} \text{ for all } m \gg 0.
 \]
This allows us to define the notion of semistability on objects.  We say that a nonzero object $E$ is \textit{$Z$-semistable} (resp.\ \textit{$Z$-stable}) if for any nonzero subobject $G \hookrightarrow E$ in $\Ap$, we have $\phi (G)(m) \leq \phi (E)(m)$ for $m \gg 0$ (resp.\ $\phi (G)(m) < \phi (E)(m)$ for $m \gg 0$).  We also write $\phi (G) \prec \phi (E)$ (resp.\ $\phi (G) \preceq \phi (E)$) to denote this.  Harder-Narasimhan filtrations for polynomial stability functions exist \cite[Section 7]{BayerPBSC}.

By \cite[Proposition 6.1.1]{BayerPBSC}, with respect to any polynomial stability function from the class above, the stable objects in $\Ap$ with $ch= (-1,0,\beta,n)$ and trivial determinant are exactly the stable pairs in Pandharipande and Thomas' paper \cite{PT}, which are 2-term complexes of the form $$[\OO_X \overset{s}{\to} F]$$ where $F$ is a pure 1-dimensional sheaf and $s$ has 0-dimensional cokernel.  For this reason, and in line with calling a stability function as above a \textit{PT-stability function} in \cite{BayerPBSC}, we call any polynomial stability condition satisfying the above requirements a \textit{PT-stability condition}, and any nonzero object in $\Ap$ semistable (resp.\ stable) with respect to it \textit{PT-semistable} (resp.\ \textit{PT-stable}).

\section{Semistable Reduction in the Derived Category}

Let us first recall how semistable reduction proceeds for flat family of sheaves  in \cite{Langton}.

Suppose $E \in \Coh (X_R)$ is a flat family of sheaves on a projective variety $X$ over the base $\Spec R$, where $R$ is a DVR over $k$.  Suppose that the generic fibre $E_K$ is $\mu$-semistable, while the central fibre $E_k$ is not.  Then there is a unique maximal destabilising quotient sheaf $E_k \twoheadrightarrow Q_0$ in the category $\Coh (X_k)$.  If we write $I^0 := E$, we can define $I^1$ to be the kernel of the composition $E \twoheadrightarrow E_k \twoheadrightarrow Q_0$, which is another flat family of sheaves on $X$ over $\Spec R$.  This process of going from the family $I^0$ to the family $I^1$ via the short exact sequence
\[ 0 \to I^1 \to I^0 \to Q_0 \to 0 \]
in $\Coh (X_R)$ is called an elementary modification.

We can now look at the central fibre of $I^1$ to see if it is $\mu$-semistable.  If it is not, we can perform another elementary modification to obtain a flat family $I^2$ over $\Spec R$, and so on.  It is the content of \cite[Theorem (2)]{Langton} (see also \cite[Theorem 2.B.1]{HL}) that this process will terminate after a finite number of steps, giving us a flat family of $\mu$-semistable sheaves over $\Spec R$.  Our aim here is to show that this phenomenon happens in the more general setting of the derived category.

\subsection{Elementary Modifications}

The process of elementary modification works not only for flat families of coherent sheaves, but also for flat families of complexes in the derived category when we fix a t-structure:

\begin{pro}[Elementary Modification in the Derived Category]\label{em}
Let $R$ be a DVR over $k$.  Given a flat family $I \in D^b(X_R)$ of objects in $\Ap$ over $\Spec R$, and a surjection $L\iota^\ast I \overset{\al}{\twoheadrightarrow} Q$ in $\Ap$, there exists a flat family $J$ of objects in $\Ap$ such that its generic fibre coincides with that of $I$, and we have an exact sequence in $\Ap$
\[
0 \to Q \to L\iota^\ast J \to L\iota^\ast I \to Q \to 0.
\]
\end{pro}

\begin{proof}
Suppose we have a surjection
\[
 L\iota^\ast I \overset{\al}{\twoheadrightarrow} Q
\]
in $\Ap$.  Then $\iota_\ast \al$ is a morphism in $D^b(X_R)$, and we can define $J \in D^b(X_R)$ by the exact triangle

\[
\def\objectstyle{\scriptstyle}
\def\labelstyle{\scriptstyle}
\xymatrix@R=1.5pc{
   J \ar[rrr]& & & I \ar[dl]^\ep \\
  & & \iota_\ast L\iota^\ast I \ar[dl]^{\iota_\ast \al} & \\
  & \iota_\ast Q \ar[uul]^{[1]}
}
\]
where $\ep$ is the adjunction map.

Applying $L\iota^\ast$ to the above exact triangle, we get another exact triangle

\[
\def\objectstyle{\scriptstyle}
\def\labelstyle{\scriptstyle}
\xymatrix@R=1.5pc{
   L\iota^\ast J \ar[rrr]& & & L\iota^\ast I \ar[dl]^{L\iota^\ast\ep} \\
  & & L\iota^\ast \iota_\ast L\iota^\ast I \ar[dl]^{L\iota^\ast\iota_\ast \al} & \\
  & L\iota^\ast \iota_\ast Q \ar[uul]^{[1]}
} .
\]

Taking the long exact sequence of cohomology with respect to the heart $\Ap$ yields
\[
0 \to  \mathcal{H}^{-1}_{\Ap} (L\iota^\ast J) \to 0  \to Q \to \mathcal{H}^0_{\Ap} (L\iota^\ast J) \to  L\iota^\ast I \overset{\al'}{\to}   Q \to \mathcal{H}^1_{\Ap} (L\iota^\ast J) \to 0
\]
where $\al' := \mathcal{H}^0_{\Ap}((L\iota^\ast \iota_\ast \al) \circ (L\iota^\ast \ep))$.

Now, if we write $\eta'$ for the adjunction map $L\iota^\ast \iota_\ast (L\iota^\ast I) \to L\iota^\ast I$, then the composition
\[
L\iota^\ast I \overset{L\iota^\ast \ep}{\longrightarrow} L\iota^\ast \iota_\ast L\iota^\ast I \overset{\eta'}{\to} L\iota^\ast I
\]
is in fact the identity map on $L\iota^\ast I$ \cite[equation (1.5.8), p.29]{CS}.  So $\HH^0 (\eta' \circ L\iota^\ast \ep)$ is the identity map on $\HH^0 (L\iota^\ast I)$.  On the other hand, using  the following exact triangle from Corollary \cite[Corollary 5.8]{Lo1}
\[
L\iota^\ast I [1] \to L\iota^\ast \iota_\ast L\iota^\ast I  \overset{\eta'}{\to} L\iota^\ast I \to L\iota^\ast I[2],
\]
we see that $\HH^0 (\eta')$ is also an isomorphism.  Hence $\HH^0 (L\iota^\ast \ep)$ itself is an isomorphism.

 Also, since we have a morphism of functors $\eta : L\iota^\ast \circ \iota_\ast \to \text{id}_{D(X_m)}$ \cite[equation (1.5.5), p.28]{CS}, we have the commutative diagram
\[
\xymatrix{
L\iota^\ast \iota_\ast L\iota^\ast I \ar[r]^(.6){\eta_{L\iota^\ast I}} \ar[d]^{L\iota^\ast \iota_\ast \al} & L\iota^\ast I \ar[d]^\al \\
L\iota^\ast \iota_\ast Q \ar[r]^{\eta_Q} & Q
}
\]
which induces a morphism of exact triangles

\[
\xymatrix{
L\iota^\ast I [1] \ar[d] \ar[r] & L\iota^\ast \iota_\ast L\iota^\ast I \ar[r] \ar[d]^{L\iota^\ast \iota_\ast \al} & L\iota^\ast I \ar[d]^\al \ar[r] & L\iota^\ast I [2] \ar[d] \\
Q[1] \ar[r] & L\iota^\ast \iota_\ast Q \ar[r] & Q \ar[r] & Q[2]
}.
\]

If we apply the functor $\HH^0$ to the above diagram, the left-most and right-most columns would vanish, while the middle horizontal maps become isomorphisms by Lemma \cite[Lemma 5.6(a)]{Lo1}.  Then, since $\al  =\HH^0(\al)$ is surjective by hypothesis, we get surjectivity of $L\iota^\ast \iota_\ast \al$.  This, combined with the fact that $\HH^0(L\iota^\ast \ep)$ is an isomorphism, which we just proved, implies $\al'$ is surjective.  Consequently, $L\iota^\ast J = \mathcal{H}^0_{\Ap} (L\iota^\ast J) \in \Ap$, meaning that $J$ is again a family of objects in $\Ap$ over $\Spec R$.
This completes the proof of the proposition.
\end{proof}

Given a nonzero object $E \in \Ap$, we can consider its HN filtration with respect to PT-stability (or, indeed, any polynomial stability)
\[
  0 \neq E_0 \hookrightarrow E_1 \hookrightarrow \cdots \hookrightarrow E_n = E
\]
where each factor $E_i/E_{i-1}$ is PT-semistable, and $\phi (E_i/E_{i-1}) \succeq \phi (E_{i+1}/E_i)$ for all $i$. Recall that the left-most factor $E_0$ in the HN filtration is called the \textit{maximal destabilising subobject}, and that it satisfies the following property \cite[Lemma 3.1]{Lo1}: for any $F \hookrightarrow E$ in $\Ap$ such that $\phi (F) \succeq \phi (E_0)$, we have $\phi (F)=\phi (E_0)$ and $F\hookrightarrow E_0$.

For semistable reduction in the derived category, we start with a flat family $I^0=I$ of objects in $\Ap$ over $\Spec R$ where $R$ is any DVR over $k$.  For any $i \geq 0$, if $L\iota^\ast I^i$ is not semistable in $\Ap$, then the proposition above says we can define a new flat family $I^{i+1}$ by the exact triangle in $D^b(X_R)$
\[
 I^{i+1} \to I^i \to \iota_\ast G^i \to I^{i+1}[1] \to \cdots
\]
where $G^i$ is the cokernel of a maximal destabilising subobject $B^i \hookrightarrow L\iota^\ast I^i$.  We also know from the proof of the proposition that we have the short exact sequence in $\Ap$
\begin{equation}\label{eqn-ssrfourterm}
0 \to G^i \to L\iota^\ast I^{i+1} \to L\iota^\ast I^i \to G^i \to 0,
\end{equation}
which we can break up into the two short exact sequences in $\Ap$
\begin{align}
  0 \to B^i \to & L\iota^\ast I^i \to G^i \to 0, \label{ses1}\\
  0 \to G^i \to &L\iota^\ast I^{i+1} \to B^i \to 0. \label{ses2}
\end{align}
We will call $I^m$ the family obtained from $I^0$ after $m$ elementary modifications.

Suppose the above process never lets us arrive at a semistable central fibre, i.e.\ suppose for all $i \geq 0$, we have a maximal destabilising subobject $B^i \hookrightarrow L\iota^\ast I^i$ such that $\phi (B^i)  \succ \phi (L\iota^\ast I^i)$.  Then we prove an analogue of \cite[Lemma 1]{Langton}:

\begin{lemma}
Using the notation above, for any $0 \neq E \subseteq L\iota^\ast I^{i+1}$, we have $\phi (E) \preceq \phi (B^i)$.  If equality holds, then $E \cap G^i =0$.
\end{lemma}

Given objects $A, B, I$ in the abelian category $\Ap$ such that $A,B \subset I$, we can construct objects $A \vee B$ and $A \cap B$ satisfying the following three properties:
\begin{itemize}
\item $A \vee B$ is a subobject of $I$.
\item We have an exact sequence
\[
 0 \to A \cap B \to A \oplus B \to A \vee B \to 0.
\]
\item We have a commutative diagram
\[
\xymatrix{
 A \arinj[r] & A \vee B \\
 A \cap B \arinj[u] \arinj[r] & B \arinj[u]
}.
\]
\end{itemize}

\begin{proof}[Proof of lemma.]
Take any $0\neq E \subset L\iota^\ast I^{i+1}$.  Let $E \vee G^i$ denote the following fibred product of subobjects of $L\iota^\ast I^{i+1}$ in $\Ap$:
\begin{equation*}
\xymatrix{
 E \arinj[r]^{j_1}  & E \vee G^i \\
E \cap G^i \ar@{^{(}->}[u] \ar@{^{(}->}[r] & G^i \arinj[u]^{j_2}
}.
\end{equation*}
We have the exact sequence
\begin{equation*}
0 \to E \cap G^i \to E \oplus G^i \to E \vee G^i \to 0 \text{\quad in $\Ap$}
\end{equation*}
and so we have
\[
  Z(E)(m) + Z(G^i)(m)  = Z ( E \vee G^i)(m) + Z(E \cap G^i)(m).
\]

From $L\iota^\ast I^{i+1}/G^i \cong B^i$, we get that $E\vee G^i$ (since it contains $G^i$) corresponds to some subobject $D \subseteq B^i$, and $E \vee G^i / G^i \cong D$.  Hence $Z(D) = Z (E \vee G^i) - Z (G^i) = Z(E) - Z(E \cap G^i)$. We have $\phi (D) \preceq \phi (B^i)$ by the semistability of $B^i$.

Since $E \cap G^i \subseteq G^i \cong L\iota^\ast I^i/B_i$, the object $E \cap G^i$ corresponds to an object $\tilde{E}$ such that $B^i \subseteq \tilde{E} \subset L\iota^\ast I^i$, and $\tilde{E}/B^i \cong E \cap G^i$.  We have the short exact sequenece
\begin{equation}\label{eqn-BitildeEE}
0 \to B^i \to \tilde{E} \to E\cap G^i \to 0
\end{equation}
where $\phi (\tilde{E}) \preceq \phi (B^i)$ since $B^i$ is defined to be the maximal destabilising subobject.  Then by the seesaw principle, $\phi (B^i) \succeq \phi (\tilde{E}) \succeq \phi  (E\cap G^i)$.  Then $\phi (E \cap G^i) \preceq \phi (B^i)$ and $\phi (D) \preceq \phi (B^i)$ together imply $\phi (E) \preceq \phi (B^i)$.

Now suppose $0 \neq E \subseteq L\iota^\ast I^{i+1}$ satisfies $\phi (E) = \phi (B^i)$.  Let $\tilde{E}$ be as above. If $E \subseteq G^i$, then $E \cap G^i = E$, and from \eqref{eqn-BitildeEE} we get $\phi (\tilde{E})= \phi (B^i)$.  Hence there is an injection $\tilde{E} \hookrightarrow B^i$ since $B^i$ is a maximal destabilising subobject.  This implies $B^i = \tilde{E}$, and so $E=E \cap G^i=0$ from \eqref{eqn-BitildeEE}, contradicting $E$ being nonzero.  Thus we must have $E \nsubseteq G^i$.

We saw above that $Z(E\cap G^i) + Z(D) = Z(E)$ and $\phi (D) \preceq \phi (B^i)$.  By assumption, we also have $\phi (E)= \phi (B^i)$.  These force $\phi (E \cap G^i) \succeq \phi (B^i)$.  On the other hand, $E\cap G^i \subseteq G^i \subseteq L\iota^\ast I^{i+1}$, and so by the first part of the proposition we have $\phi (E \cap G^i) \preceq \phi (B^i)$.  Hence $\phi (E\cap G^i) = \phi (B^i)$, which forces $E\cap G^i=0$ by the argument in the previous paragraph.
\end{proof}

\begin{coro}\label{Bi1leqBi}
$\phi (B^{i+1}) \preceq \phi (B^i)$, with equality only if $B^{i+1} \cap G^i = 0$.
\end{coro}

Suppose $X$ is a smooth projective three-fold over $k$.  In showing the valuative criterion for universal closedness for PT-semistable objects, we need to take a PT-semistable object $I^0_K \in \Ap_K$ with $ch(L\iota^\ast I^0_K)=(-r,-d,\beta,n)$ where $r>0$.  Then we need to extend $I^0_K$ to an $R$-flat family $I^0 \in D^b(X_R)$ of objects in $\Ap$, and then apply semistable reduction to $I^0$.  We now consider the properties the central fibre $L\iota^\ast I^0$ can have when we extend $I^0_K$ to $I^0$.

\begin{pro}\label{pro-coh31ssreduction}
Let $X$ be a smooth projective three-fold over $k$.  If $E_K$ is a torsion-free sheaf on $X_K$ that is semistable in $\Coh_{3,1}(X_K)$, then there is an $R$-flat coherent sheaf $E$ on $X_R$ such that $j^\ast E = E_K$, and $\iota^\ast E$ is a torsion-free sheaf semistable in $\Coh_{3,1}(X_k)$.
\end{pro}

\begin{proof}
Take any $R$-flat coherent sheaf $E$ on $X_R$ that restricts to $E_K$ on $X_K$.  By \cite[Theorem 2.B.1]{HL}, replacing $E$ by a subsheaf if necessary, we can assume that $\iota^\ast E$ is also semistable in $\Coh_{3,1}$.  Suppose $\iota^\ast E$ is not torsion-free.  Then it has a maximal torsion subsheaf $T \subset \iota^\ast E$, which is necessarily 0-dimensional.  And we can run semistable reduction on the family $E$; the details are as follows:

Define $E'$ as the kernel of the composition $E \twoheadrightarrow \iota_\ast \iota^\ast E \twoheadrightarrow \iota_\ast (\iota^\ast E/T)$.  Pulling back the short exact sequence $0 \to E' \to E \to \iota_\ast (\iota^\ast E/T) \to 0$ to $X_k$ and taking the long exact sequence of Tor, we get
\[
0 \to \iota^\ast E/T  \to \iota^\ast E' \to \iota^\ast E \to \iota^\ast E/T \to 0.
\]
Let $T'$ be the maximal torsion subsheaf of $\iota^\ast E'$.  Suppose $T'$ is nonzero.  Then we must have $(\iota^\ast E/T )\cap T' = 0$, or else $\iota^\ast E/T$ would have a torsion subsheaf, contradicting the maximality of $T$.  Continuing this process, we obtain a sequence of $R$-flat coherent sheaves $\cdots \to E^{(n)} \to E^{(n-1)} \to \cdots \to E' \to E$, such that if $T^{(n)}$ denotes the maximal torsion subsheaf of $\iota^\ast E^{(n)}$, then $T^{(n)} \hookrightarrow T^{(n-1)}$ for all $n$.  If we do not arrive at some $n$ with $T^{(n)}=0$, then because all the $T^{(n)}$ are 0-dimensional, we have $T^{(n)} \cong T$ for some $T$ for $n \gg 0$.  Then, by the same argument as in the proof of \cite[Theorem 2.B.1]{HL}, this produces a nonzero 0-dimensional subsheaf of $E_K$, a contradiction.  The proposition follows.
\end{proof}

The following is a strengthening of the $d=0$ case of \cite[Theorem 4.1]{Lo1}:

\begin{pro}\label{pro-Lo1theorem41improv}
Let $X$ be a smooth projective three-fold over $k$.  Given any object
\[
  E_K \in \langle \Coh_{\leq 0}(X_K),\Coh_{\geq 3}(X_K)[1]\rangle \subset D^b(X_K),
\]
there exists a 3-term complex $\widetilde{E}^\bullet$ of $R$-flat coherent sheeaves with $R$-flat cohomology on $X_R$ such that:
\begin{itemize}
\item the generic fibre $j^\ast (\widetilde{E}^\bullet) \cong E_K$ in $D^b(X_K)$;
\item for the central fibre $L\iota^\ast (\widetilde{E}^\bullet)$, the cohomology $H^0(L\iota^\ast (\widetilde{E}^\bullet))$ is 0-dimensional, and the cohomology $H^{-1}(L\iota^\ast (\widetilde{E}^\bullet))$ is semistable in $\Coh_{3,1}(X)$.
\end{itemize}
\end{pro}
\begin{proof}
The proof of \cite[Theorem 4.1]{Lo1} still works, except that we now extend $\kernel (s_K)$ in the proof to an $R$-flat family of torsion-free sheaves that are semistable in $\Coh_{3,1}$ using Proposition \ref{pro-coh31ssreduction}.
\end{proof}

Hence if we start with a PT-semistable object $I^0_K \in \Ap (X_K)$ of nonzero rank, we can extend it to an $R$-flat family $I^0$ of objects in $\Ap$ using Proposition \ref{pro-Lo1theorem41improv}, and apply elementary modifications starting from $I^0$.  The following proposition describes what happens when we do not arrive at a semistable central fibre after a finite number of modifications:

\begin{pro}\label{infiniteemBGconstant}
Suppose $I^0$ is an $R$-flat family of objects of nonzero rank in $\Ap$, where the generic fibre $j^\ast (I^0)$ is PT-semistable, and the central fibre $L\iota^\ast I^0$ is such that $H^0(L\iota^\ast I^0)$ is 0-dimensional, and $H^{-1}(L\iota^\ast I^0)$ is semistable in $\Coh_{3,1}(X_k)$.  Suppose that we do not arrive at a PT-semistable central fibre $L\iota^\ast I^j$ after a finite number of elementary modifications starting from $I^0$. Then the morphisms $B^{i+1} \to B^i$ and $G^{i+1} \to G^i$ induced by $L\iota^\ast I^{i+1} \to L\iota^\ast I^i$ eventually become isomorphisms.  Say they are isomorphic to $B$ and $G$.  Then we have $L\iota^\ast I^i = B \oplus G$ in $\Ap$ for all $i \gg 0$.
\end{pro}

\begin{proof}
Since $\phi (B^i)  \succ \phi (L\iota^\ast I^i)$ for all $i$, the degree of the polynomial $Z(B^i)(m)$ must be 2, 0 or 3.  Since $\phi (B^{i+1}) \preceq \phi (B^i)$ for all $i$, $\degree Z(B^i)(m)$ is constant for $i \gg 0$.  For instance, if $Z(B^i)(m)$ is a degree-2 polynomial, then $Z(B^{i+1})(m)$ could be of degree 2, 0 or 3; but if $Z(B^i)(m)$ is a degree-0 polynomial, then $Z(B^{i+1})(m)$ can only be of degree 0 or 3.

Note that, since $H^{-1}(L\iota^\ast I^0)$ is nonzero and torsion-free, if $H^{-1}(B^0)$ is nonzero then it is also torsion-free, in which case $B^0$ has 3-dimensional support.  By Corollary \ref{Bi1leqBi}, none of the subsequent $B^i$ can be 2-dimensional.  Hence all $B^i$ are 0-dimensional or 3-dimensional.  One consequence of this is that  $H^{-1}(L\iota^\ast I^i)$ is torsion-free for all $i$: for, if not, then $H^{-1}(L\iota^\ast I^i)$ has a 2-dimensional subsheaf $T$ for some $i$, and $T[1] \hookrightarrow L\iota^\ast I^i$ would be destabilising, and has a larger $\phi$ than $B^i$, a contradiction.

If $\degree Z(B^i)(m)=0$ for  $i \gg 0$, then $B^i$ is supported on dimension 0 for $i \gg 0$, and so  $\phi (B^i)$ would be constant.  By Corollary \ref{Bi1leqBi}, $0=B^{i+1} \cap G^i = B^{i+1} \cap \kernel (L\iota^\ast I^{i+1} \to B^i) = \kernel (B^{i+1} \to B^i)$.  Hence we have injections $B^{i+1} \hookrightarrow B^i$ for $i \gg 0$.  So $B^i = H^0(B^i)$ must stabilise eventually.

From now on, assume that $\degree Z(B^i)(m)=3$ for $i \gg 0$.  Let $ch_d (B^i)=\beta_d^i$ and $ch_d(G^i)=\gamma_d^i$ for $0 \leq d \leq 3$.  Then from the exact sequence $0 \to B^i \to L\iota^\ast I^i \to G^i \to 0$, and because $ch(L\iota^\ast I^i)$ is the same for all $i$, we get the equalities
\begin{align*}
  -r &= \beta_3^i + \gamma_3^i \\
  -d &= \beta_2^i + \gamma_2^i \\
  \beta &= \beta_1^i + \gamma_1^i \\
  n &= \beta_0^i + \gamma_0^i.
\end{align*}
From $B^i \hookrightarrow L\iota^\ast I^i$, we get $H^{-1}(B^i) \subset H^{-1}(L\iota^\ast I^i)$.  Hence $1 \leq ch_0(H^{-1}(B^i)) = -\beta_3^i \leq r$.  So $-r \leq \beta_3^i \leq -1$ for all $i$.

Since $\phi (B^i) \succ \phi (L\iota^\ast I^i)$, we have $\mu (H^{-1}(B^i)) \geq \mu (H^{-1}(L\iota^\ast I^i))$, and so $\frac{-\beta^i_2}{-\beta^i_3} \geq \frac{d}{r}$, which is fixed.  Hence the sequence $\left( \frac{-\beta_2^i}{-\beta_3^i}\right)$ is a decreasing sequence (by Corollary \ref{Bi1leqBi}) in $\frac{1}{r!}\ZZ$ bounded from below, and so must eventually become constant.  We show that, in fact, the phase $\phi (B^i)$ itself must eventually become constant.

Suppose $B^i$ is 0-dimensional for $0 \leq i \leq i_0$, and 3-dimensional for $i \geq i_0+1$ (we allow $i_0+1$ to be 0).  Recall the short exact sequences \eqref{ses1}, \eqref{ses2} in $\Ap$ for $i \geq 0$:
\begin{gather*}
  0 \to B^i \to L\iota^\ast I^i \to G^i \to 0 \\
  0 \to G^i \to L\iota^\ast I^{i+1} \to B^i \to 0.
  \end{gather*}
By hypothesis, $H^{-1}(L\iota^\ast I^0)$ is semistable in $\Coh_{3,1}(X_k)$.  From the exact sequence
\[ 0 \to H^{-1}(L\iota^\ast I^0) \to H^{-1}(G^0) \to H^0(B^0) \]
we get that $H^{-1}(G^0)$ is also semistable in $\Coh_{3,1}$.

For $0 \leq i \leq i_0$, we know $B^i$ is 0-dimensional, and so we have the exact sequence $0 \to H^{-1}(L\iota^\ast I^i) \to H^{-1}(G^i) \to H^0(B^i)$, and $H^{-1}(G^i) \cong H^{-1}(L\iota^\ast I^{i+1})$.  Repeating the argument in the previous paragraph, we get that $H^{-1}(L\iota^\ast I^{i_0+1})$ is semistable in $\Coh_{3,1}$.

When $0 \leq i \leq i_0$, we have $H^0(B^i)=B^i$ is 0-dimensional; when $i  \geq i_0+1$, we have that $B^i$ is 3-dimensional but PT-semistable, so  $H^0(B^i)$ is 0-dimensional by \cite[Lemma 3.3]{Lo1}.  So for every $i \geq 0$, we have that $H^0(B^i)$ is 0-dimensional.  This, along with $H^0(L\iota^\ast I^0)$ being 0-dimensional, implies that  $H^0 (L\iota^\ast I^i)$ is  0-dimensional for all $i \geq 0$.  Since all the $L\iota^\ast I^i$ have the same Hilbert polynomial, this implies that all $H^{-1} (L\iota^\ast I^i)$ have the same reduced Hilbert polynomial $p_{3,1}$.

Now, $H^{-1}(B^{i_0+1})$ is a nonzero torsion-free subsheaf of  $H^{-1}(L\iota^\ast I^{i_0+1})$, which is semistable in $\Coh_{3,1}$.  However,  $\phi (B^{i_0+1}) \succ \phi (L\iota^\ast I^{i_0+1})$.  Therefore, the two reduced Hilbert polynomials $p_{3,1} (H^{-1}(B^{i_0+1}))$ and $p_{3,1} (H^{-1}(L\iota^\ast I^{i_0+1}))$ are forced to be equal.

That $\phi (B^{i+1}) \preceq \phi (B^i)$ for all $i$ implies $p_{3,1} (H^{-1}(B^{i+1})) \preceq p_{3,1} (H^{-1}(B^i))$ for all $i \geq i_0+1$.  Also, $\phi (H^{-1}(B^i)) \succ \phi (H^{-1}(L\iota^\ast I^i))$ for all $i \geq i_0+1$, hence $p_{3,1} (H^{-1}(B^i)) \succeq p_{3,1} (H^{-1}(L\iota^\ast I^i))$ for all $i \geq i_0+1$.  These inequalities, together with the conclusions of the last two paragraphs, imply that $p_{3,1} (H^{-1}(B^i)) = p_{3,1} (H^{-1}(L\iota^\ast I^i))$ is the same for all $i \geq i_0+1$.

That $\phi (B^{i+1}) \preceq \phi (B^i)$ means that the sequence $\{ch_3(B^i)/\rank (H^{-1}(B^i))\}_{i \geq i_0+1}$ is a decreasing sequence.  On the other hand, that $\phi (B^i) \succ \phi (L\iota^\ast I^i)$ for all $i$ and that $p_{3,1} (H^{-1}(B^i)) = p_{3,1} (H^{-1}(L\iota^\ast I^i))$ for all $i \geq i_0+1$ together imply
\[
\frac{ch_3(L\iota^\ast I^0)}{\rank (H^{-1}(L\iota^\ast I^0))} \leq \frac{ch_3(B^i)}{\rank (H^{-1}(B^i))}
 \]
 for all $i \geq i_0+1$.  Since $\rank (H^{-1}(B^i))$ is a positive integer at most equal to $\rank (H^{-1}(L\iota^\ast I^0))$ for $i \geq i_0+1$, the decreasing sequence $\{ch_3(B^i)/\rank (H^{-1}(B^i))\}_{i \geq i_0+1}$ is bounded below by the constant $ch_3(L\iota^\ast I^0)/\rank (H^{-1}(L\iota^\ast I^0))$.  So $ch_3(B^i)/\rank (H^{-1}(B^i))$ must stabilise for large enough $i$, at which point $\phi (B^i)$  becomes constant.

By Corollary \ref{Bi1leqBi}, $0=B^{i+1} \cap G^i = B^{i+1} \cap \kernel (L\iota^\ast I^{i+1} \to B^i) = \kernel (B^{i+1} \to B^i)$ (where the map $L\iota^\ast I^{i+1} \to B^i$ is from the sequence \eqref{ses2}).  That is, we have injections $B^{i+1} \hookrightarrow B^i$ in $\Ap$.  Suppose this injection has cokernel $C$.  From the long exact sequence of cohomology of $0 \to B^{i+1} \to B^i \to C \to 0$, and the fact that the cohomolgy of $B^i$ becomes constant at each degree for $i \gg 0$ (compare Hilbert polynomials to see this), we get that each cohomology of $C$ must be zero for large enough $i$.  In other words, for $i \gg 0$, the map $L\iota^\ast I^{i+1} \to B^i$ takes $B^{i+1}$ isomorphically onto $B^i$.  That is, the map $L\iota^\ast I^{i+1} \to L\iota^\ast I^i$ in the sequence \eqref{eqn-ssrfourterm} takes $B^{i+1}$ isomorphically onto $B^i$.  This induces isomorphisms on the cokernels of the inclusions $B^i \hookrightarrow L\iota^\ast I^i$ and $B^{i+1} \hookrightarrow L\iota^\ast I^{i+1}$, giving us a commutative diagram for $i \gg 0$
\[
\xymatrix{
B^{i+1} \arinj[r] \ar[d]^\thicksim & L\iota^\ast I^{i+1} \arsurj[r] \ar[d] & G^{i+1} \ar[d]^\thicksim \\
B^i \arinj[r] & L\iota^\ast I^i \arsurj[r] & G^i
}.
\]
Let us fix objects $B,G$ that are isomorphic to  $B^i, G^i$ for all $i \gg 0$, respectively.

Now, if we take the injection $f_1 : G^i \hookrightarrow L\iota^\ast I^{i+1}$ from \eqref{eqn-ssrfourterm} and compose it with the surjection $f_2 : L\iota^\ast I^{i+1} \twoheadrightarrow G^{i+1}$ from \eqref{ses1}, the kernel of the composition $f_2f_1 : G^i \hookrightarrow L\iota^\ast I^{i+1} \twoheadrightarrow G^{i+1}$ would be isomorphic to $G^i \cap B^{i+1}$ (since the kernel of the surjection $L\iota^\ast I^{i+1} \twoheadrightarrow G^{i+1}$ is $B^{i+1}$), which we just showed is zero for $i \gg 0$.  Thus the composition $f_2f_1 : G^i \to G^{i+1}$ is an injection.  Since $G^i$ and $G^{i+1}$ are isomorphic for $i \gg 0$, $f_2f_1$ is an isomorphism  for $i \gg 0$.  Hence the map $f_1(f_2f_1)^{-1}$ splits the short exact sequence \eqref{ses1}, and so $L\iota^\ast I^i \cong B^i \oplus G^i \cong B \oplus G$ for  $i \gg 0$.
\end{proof}

\subsection{Constructing an Inverse System}

For the rest of this section, assume the setup of Proposition \ref{infiniteemBGconstant}.  That is, start with a PT-semistable object $I^0_K \in \Ap_K$, extend it to a flat family $I:=I^0$ of objects in $\Ap$ over $\Spec R$, assume that the central fibre of $I^0$ is not PT-semistable, and that a finite number of semistable reduction does not yield a PT-semistable central fibre.  This gives us an infinite sequence of flat families
$$ \cdots \to I^2 \to I^1 \to I^0$$
in $D^b(X_R)$.  We would now like to show that, if $R$ is a complete DVR, then this would imply that $j^\ast I^0$ is PT-unstable, a contradiction.

\subsubsection{The Cohomology of Various Objects}

Using the techniques developed in \cite{Lo1}, we now compute the cohomology of various objects.

\noindent\textit{Definition of $Q^m$.} By Proposition \ref{infiniteemBGconstant}, we can replace $I^0$ by $I^m$ for $m \gg 0$ if necessary, and assume that all the $B^i$ and $G^i$ are isomorphic, to some $B$ and $G$, respectively.  We can pull back the composition $I^m \to I^{m-1} \to \cdots \to I^0=I$ from $D(X_R)$ to $D(X_m)$, obtaining a morphism in $\Ap_m$ which we will denote by $\theta_m$:
\[
\theta_m : L\iota_m^\ast I^m \to L\iota_m^\ast I.
\]

Also, we define $Q^m \in D(X_R)$ by the exact triangle
\[
I^m \to I \to Q^m \to I^m[1].
\]
Since all the $I^i$ have bounded cohomology, we know $Q^m$ also has bounded cohomology, i.e. $Q^m  \in D^b(X_R)$.

We can also apply the octrahedral axiom to the commutative triangle in $D(X_R)$
\[
\xymatrix{
  & I^{m-1} \ar[dr] & \\
I^m \ar[ur] \ar[rr] & & I^0
}
\]
and obtain
\begin{equation}\label{octahedral1}
\xymatrix{
& & & \iota_\ast G \ar[ddd] \\
& & & \\
& I^{m-1} \ar[dr] \ar[uurr] & & \\
I^m \ar[rr] \ar[ur] & & I^0 \ar[r] \ar[dr] & Q^m  \ar[d] \\
& & & Q^{m-1}
}.
\end{equation}

In particular, we obtain the exact triangle in $D(X_R)$ for every $m\geq 1$:
\[
\iota_\ast G \to Q^m \to Q^{m-1} \to \iota_\ast G[1].
\]

Modifying the proof of \cite[Lemma 2.4.1]{AP} slightly, we get the following lemma that says $Q^m$ lives in  $D(X^m)$:

\begin{lemma}\label{XmXnXm+n}
Let $F \in D^b(X_m)$ and $G \in D^b(X_n)$.  Then for every morphism ${\iota_m}_\ast F \overset{\al}{\to} {\iota_n}_\ast G$ in $D(X_R)$, there exists a morphism ${\iota_{m+n,m}}_\ast F \overset{\beta}{\to} {\iota_{m+n,n}}_\ast G$ in $D^b(X_{m+n})$ such that $\al = {\iota_{m+n}}_\ast \beta$.
\end{lemma}
\begin{proof}
We adapt the proof of \cite[Lemma 2.4.1]{AP} to our situation.  Let $F^\bullet, G^\bullet$ be bounded complexes of coherent sheaves on $X_m, X_n$, representing $F,G$, respectively.  We can choose a bounded complex $P^\bullet$ of torsion-free $\OO_{X_R}$-modules, and a surjective morphism of complexes of $\OO_{X_R}$-modules $q : P^\bullet \to F^\bullet$ that is a quasi-isomorphism, and a chain map $p : P^\bullet \to G^\bullet$ such that $\al = pq^{-1}$ in $D(X_R)$:
\begin{equation*}
\xymatrix{
 & P^\bullet \ar[dl]_q^{\thicksim} \ar[dr]^p & \\
F^\bullet \ar[rr]^\al & & G^\bullet
}.
\end{equation*}

Define $K^\bullet := \kernel (q)$ in the abelian category $\mbox{Kom}(X_R)$.  Then $K^\bullet$ is acyclic (i.e.\ all cohomology are 0), and multiplication by $\pi^n$ is injective on each $K^i$.  Set $\overline{P^\bullet} := P^\bullet /\pi^n K^\bullet$ (quotient in $\mbox{Kom}(X_R)$).  Then $q$ factors through $\overline{q} : \overline{P^\bullet} \to F^\bullet$ because $\pi^n K^\bullet \subseteq K^\bullet$.  Also, $p$ factors through $\overline{p} : \overline{P^\bullet} \to G^\bullet$ because $\pi^n K^\bullet \subset \pi^n P^\bullet$ and $p(\pi^n P^\bullet)=0$ (since the $G^i$ live on $X_n$).  Now $\overline{q}$ is still a quasi-isomorphism, so $\al = \overline{p}\overline{q}^{-1}$.

However, because $\pi^m P^\bullet \subset K^\bullet \subset P^\bullet$, we get that $P^\bullet/\pi^{m+n}P^\bullet \twoheadrightarrow P^\bullet/\pi^n K^\bullet = \overline{P^\bullet}$ in $\mbox{Kom} (X_R)$.  That is, $\overline{P^\bullet}$ is a complex of $\OO_{X_{m+n}}$-modules.

Hence $\overline{q}, \overline{p}$ can both be considered as chain maps over $X_{m+n}$, and so $\al$ is the pushforward of a morphism in $D^b(X_{m+n})$.
\end{proof}

\begin{coro}\label{QmfromXm}
$Q^m \cong {\iota_m}_\ast \tilde{Q}^m$ for some $\tilde{Q}^m \in D^b(X_m)$.
\end{coro}
\begin{proof}
We have the exact triangle in $D(X_R)$
\[
\iota_\ast G \to Q^m \to Q^{m-1} \to \iota_\ast G[1] .
\]
So $Q^m$ is the cone of some morphism $Q^{m-1}[-1] \overset{\al}{\to} \iota_\ast G$ in $D(X_R)$.  By induction and the lemma above, $\al={\iota_m}_\ast \beta$ for some morphism $\beta$ in $D(X_m)$.  Let $\tilde{Q}^m := \cone (\beta)$ in $D(X_m)$.  Then ${\iota_m}_\ast \tilde{Q}^m$ must be quasi-isomorphic to $Q^m$ in $D(X_R)$.  That $\tilde{Q}^m$ has bounded cohomology follows from $Q^m$ itself having bounded cohomology, and the exactness of pushforward ${\iota_m}_\ast$.
\end{proof}

\begin{lemma}\label{cohomLmIrLmQr}
With respect to the t-structure with heart $\Ap_m$, and for $r \geq 0$, $m \geq 1$,
\begin{itemize}
\item[(a)] $\mathcal{H}^i L\iota_m^\ast I^r =0$ whenever $i\neq 0$.
\item[(b)] $\mathcal{H}^i L\iota_m^\ast Q^r =0$ whenever $i\neq 0, -1$.
\end{itemize}
In particular, $L\iota_m^\ast I^r \in \Ap_m$ for all $m \geq 1, r \geq 0$.
\end{lemma}

\begin{proof}
By assumption, $L\iota^\ast I^0 \in \Ap$, so (a) follows from Proposition \cite[Proposition 5.10]{Lo1}.  As for the cohomology of $Q^r$, we start with the exact triangle
\[
I^r \to I^0 \to Q^r \to I^r[1]
\]
which can be pulled back to
\begin{equation}\label{LmIIQtriangle}
L\iota_m^\ast I^r \to L\iota_m^\ast I^0 \to L\iota_m^\ast Q^r \to L\iota_m^\ast I^r[1].
\end{equation}
Then we can take the long exact sequence of cohomology with respect to the heart $\Ap_m$, and use part (a) to conclude (b).
\end{proof}

\begin{lemma}
The object $\tilde{Q}^m \in D^b(X_m)$ lies in the heart $\Ap_m$.
\end{lemma}

\begin{proof}
From the long  exact sequence of cohomology of \eqref{LmIIQtriangle} with respect to $\Ap_r$ when  $r=m$, we see that $\HH^i_{\Ap_m} (L\iota_m^\ast \iotam_\ast \tilde{Q}^m)$ is nonzero only for $i= 0, -1$.  Since
\[
\iotam_\ast L\iota_m^\ast \iotam_\ast \tilde{Q}^m \cong \iotam_\ast \tilde{Q}^m[1] \oplus \iotam_\ast \tilde{Q}^m
\]
by \cite[Corollary 5.8]{Lo1}, at least we know  $\tilde{Q}^m$ lies in $D^{\leq 0}_{\Ap_m}(X_m)$.

Now, suppose $i$ is the smallest integer such that $\HH^i (\tilde{Q}^m)\neq 0$.  Suppose $H^{i-1}(\tilde{Q}^m)\neq 0$; then this cohomology lies in $\Coh_{\geq 2}(X_m)$.  Also, $H^{i-1}(\iotam_\ast \tilde{Q}^m) \neq 0$.  Using the isomorphism \[\iotam_\ast L\iota_m^\ast \iotam_\ast \tilde{Q}^m \cong \iotam_\ast \tilde{Q}^m[1] \oplus \iotam_\ast \tilde{Q}^m\] again, we see that $H^{i-2}(\iotam_\ast L\iota_m^\ast \iotam_\ast \tilde{Q}^m) \neq 0$ as well, and so $H^{i-2}(L\iotam_\ast \iotam_\ast \tilde{Q}^m)$ is nonzero and lies in $\Coh_{\geq 2}(X_m)$.  Hence $\HH^{i-1}(L\iota_m^\ast \iotam_\ast \tilde{Q}^m) \neq 0$.

On the other hand, suppose $H^{i-1}(\tilde{Q}^m)=0$.  Then  $H^i(\tilde{Q}^m)\neq 0$, and  $H^i(\tilde{Q}^m)$  lies in $\Coh_{\leq 1}(X_m)$.  As above, this means that $\HH^{i-1}(\iotam_\ast L\iota_m^\ast \iotam_\ast (\tilde{Q}^m))$ is nonzero and lies in $\Coh_{\leq 1}(X_m)$.  Therefore $\HH^{i-1}(L\iota_m^\ast \iotam_\ast \tilde{Q}^m) \neq 0$.

We can thus conclude that if $i$ is the least integer for which $\HH^i(\tilde{Q}^m)$ is nonzero, then the cohomology $\HH^{i-1}(L\iota_m^\ast \iotam_\ast \tilde{Q}^m)$ is nonzero.  However, by the previous lemma, we know $\HH^j(L\iota_m^\ast \iotam_\ast \tilde{Q}^m)$ is nonzero only for $j=0$ or $-1$.  So $i$ must be greater than or equal to 0.  Since we already know $\tilde{Q}^m \in D^{\leq 0}_{\Ap_m}(X_m)$, we can conclude that $\tilde{Q}^m$ actually lies in $\Ap_m$.
\end{proof}

\begin{lemma}\label{objectheartLiicohom}
For any $m \geq 1$, and any $F \in\Ap_m$, we have $\HH^i (L\iota_m^\ast \iotam_\ast F)=0$ for $i\neq 0, -1$.  For $i=0$ or $-1$, $\HH^i (L\iota_m^\ast \iotam_\ast F) \cong F$.
\end{lemma}

\begin{proof}
This follows from \cite[Corollary 5.8]{Lo1}.
\end{proof}

Applying $\HH^0_{\Ap} \circ L\iota^\ast$ to the  diagram (\ref{octahedral1}), we obtain the following commutative diagram, where each straight line is exact at the middle term.  Because the $I^j$ are flat families, the morphisms labelled $\gamma_m, \gamma_{m-1}$ and $\mu_m$ are  surjective in $\Ap$:
\begin{equation}\label{octahedral2}
\xymatrix{
  & & & \HH^0 L\iota^\ast \iota_\ast G \ar[ddd]  \\
  & & &  \\
  & L\iota^\ast I^{m-1} \ar[dr] \ar[uurr] & & \\
 L\iota^\ast I^m \ar[rr] \ar[ur] & & L\iota^\ast I^0 \ar[r]^{\gamma_m} \ar[dr]_{\gamma_{m-1}} & \HH^0 L\iota^\ast Q^m  \ar[d]^{\mu_m} \\
 & & & \HH^0 L\iota^\ast Q^{m-1}
}.
\end{equation}

\begin{lemma}\label{allLiotaQmisom}
For $m>1$, the morphism $\mu_m := \HH^0 L\iota^\ast (Q^m \to Q^{m-1})$ is an isomorphism, and $\HH^0L\iota^\ast Q^m\cong G$ for all $m >0$.
\end{lemma}

\begin{proof}
 We are assuming that the central fibres $I^i$ have stabilised, i.e.\  $L\iota^\ast I^i \cong B \oplus G$ for each $i$, and that the pullback of each elementary modification $I^m \to I^{m-1}$ to $D(X_k)$ looks like
\begin{align*}
  L\iota^\ast I^m \cong B\oplus G   &\to    B \oplus G \cong L\iota^\ast I^{m-1},
\end{align*}
which takes $B$ isomorphically onto $B$ and $G$ to 0.  Hence $\gamma_m, \gamma_{m-1}$ are surjective with the same kernel.  Thus $\mu_m$ is an isomorphism for each $m>1$.  That the central fibres have stabilised also means $Q^1 = \iota_\ast G$.  Therefore, $\HH^0 L\iota^\ast Q^1 \cong G$.  By the previous paragraph, $\HH^0 L\iota^\ast Q^m \cong G$ for all $m>1$.
\end{proof}

\subsubsection{The Inverse System}\label{sssec-invsysconst}

In order to construct a morphism $I^0 \to Q_R$ in $D^b(X_R)$ such that its pullback to $X_K$ is a destabilising quotient, we need to show that the surjective morphisms in $\Ap_m$
\[
\xymatrix{
  L\iota_m^\ast I^0 \arsurj[r]^(.4){\al_m} & \HH^0 L\iota_m^\ast Q^m
  }
\]
form an inverse system with respect to the functors $\HH^0_{\Ap_{m'}} \circ L\iota^\ast_{m,m'} : \Ap_m \to \Ap_{m'}$.  Let us take a moment to clarify what this means.  Apply the functor $\HH^0_{\Ap_m}  L\iota^\ast_m$ to diagram \eqref{octahedral1}, and the lower right corner of the diagram becomes
\begin{equation}
\xymatrix{
L\iota_{m}^\ast I^0 \ar@{>>}[r]^(.4){\al_m} \ar@{>>}[dr] & \HH^0 L\iota_{m}^\ast  Q^m  \ar@{>>}[d]  \\
& \HH^0 L\iota_{m}^\ast Q^{m-1}
}.
\end{equation}

If we further apply the functor $\HH^0 L\iota_{m,m-1}^\ast$, we obtain
\begin{equation}
\xymatrix{
L\iota_{m-1}^\ast I^0 \ar@{>>}[rr]^{\HH^0L\iota_{m,m-1}^\ast\al_m} \ar@{>>}[drr]^{\al_{m-1}} & & \HH^0 L\iota_{m-1}^\ast  Q^m  \ar@{>>}[d]^{q_{m-1}}  \\
& & \HH^0 L\iota_{m-1}^\ast Q^{m-1}
}
\end{equation}
where we define the vertical map to be $q_{m-1}$.  Here we are using the isomorphism of functors $$\HH^0 L\iota_{m'}^\ast \cong \HH^0 L\iota_{m,m'}^\ast \HH^0 L\iota^\ast_m$$ for $1 \leq m' < m$.  Repeating this construction for $\al_{m-1}$, we obtain a commutative diagram
\begin{equation}
\xymatrix{
L\iota_{m-2}^\ast I^0 \ar[rrr]^{\HH^0 L\iota_{m,m-2}^\ast \al_m} \ar[ddrrr]^(.5){\HH^0 L\iota_{m-1,m-2}^\ast \al_{m-1}} \ar[dddrrr]_{\al_{m-2}} && & \HH^0 L\iota_{m-2}^\ast Q^m \ar[dd]^{\HH^0 L\iota_{m-1,m-2}^\ast q_{m-1}} \\
& & & \\
& && \HH^0 L\iota_{m-2}^\ast Q^{m-1} \ar[d]^{q_{m-2}} \\
& && \HH^0 L\iota_{m-2}^\ast Q^{m-2}
}.
\end{equation}

Now we see that, if we can show that each $q_{m-1}$ is an isomorphism, then we would have the inverse system described above.

By Lemma \ref{allLiotaQmisom}, $q_1$ is an isomorphism (since $q_1$ is just $\mu_2$ in the notation of that lemma).  To show that $q_m$ is an isomorphism for all $m \geq 1$, we make use of the following:

\begin{lemma}
If $\tilde{Q}^{m-1}$ is a flat family of objects in $\Ap$ over $\Spec R/\pi^{m-1}$, then $q_{m-1}$ is an isomorphism.
\end{lemma}
\begin{proof}
Define $K$ to be the kernel of $q_{m-1}$ in $\Ap_{m-1}$, so that we have a short exact sequence
\[0 \to K \to \HH^0 L\iota_{m-1}^\ast Q^m \overset{q_{m-1}}{\longrightarrow} \HH^0 L\iota_{m-1}^\ast Q^{m-1} \cong \tilde{Q}^{m-1} \to 0. \]
Pulling back this exact triangle via $L\iota_{m-1,1}^\ast$, and then applying the functor $\HH^0_{\Ap}$, we get the exact sequence in $\Ap$
$$ 0 \to \HH^0 L\iota_{m-1,1}^\ast K \to \HH^0 L\iota^\ast Q^m \to L\iota_{m-1,1}^\ast \tilde{Q}^{m-1} \to 0$$
(this is where we use the flatness assumption on $\tilde{Q}^{m-1}$).  Since $\HH^0 L\iota^\ast Q^m$ and $L\iota_{m-1,1}^\ast \tilde{Q}^{m-1}$ are both isomorphic to $G$, the surjection between them is an isomorphism, forcing the cohomology $\HH^0 L\iota_{m-1,1}^\ast K$ to be zero.  By \cite[Lemma 5.5]{Lo1}, $K=0$.  That is, $q_{m-1}$ is an isomorphism.
\end{proof}

This lemma reduces the problem of constructing an inverse system to the problem of showing the flatness of $\tilde{Q}^m$ over $\Spec R/\pi^m$ for all $m \geq 1$.  We can solve the latter problem in two steps: (1) show that, for $m \geq 1$, if $\tilde{Q}^m$ is flat then $\tilde{Q}^{2m}$ is also flat; (2) show that, for $m\geq 3$, if $\tilde{Q}^m$ is flat, then $\tilde{Q}^{m-1}$ is also flat.  Since $\tilde{Q}^1 \cong G$ is trivially flat over $\Spec R/\pi$, we will then have flatness for all $\tilde{Q}^m$.

\subsubsection{Flatness of $\tilde{Q}^m$ in General}

\begin{lemma}\label{Qmflatness-lemma1}
Let $m > 2$.  Suppose $A,B$ are objects in $\Ap_{m-1}$ and $\Ap_1$, respectively.  Given a morphism $\bar{\beta} : A \to {\iota_{m-1,1}}_\ast B$, if $\HH^0 L\iota_{m-1,1}^\ast \bar{\beta}=0$, then $\bar{\beta}$ itself is 0.
\end{lemma}

\begin{proof}
Since the truncation $\tau^{\geq 0}$ is left adjoint to the embedding functor, we have a morphism of functors $\mbox{id} \to \tau^{\geq 0}_{\Ap_{m-1}} : D(X_{m-1})  \to D^{\geq 0}_{\Ap_{m-1}}(X_{m-1})$.  This yields a commutative diagram
\begin{equation*}
\xymatrix{
  L\iota_{m-1,1}^\ast A \ar[rr]^(.4){L\iota_{m-1,1}^\ast \bar{\beta}} \ar[d] && L\iota_{m-1,1}^\ast {\iota_{m-1,1}}_\ast B \ar[rr]^(.6)\eta \ar[d] && B \ar[d]^{=} \\
  \HH^0 L\iota_{m-1,1}^\ast A \ar[rr]^(.4){\HH^0 L\iota_{m-1,1}^\ast \bar{\beta}} && \HH^0 L\iota_{m-1,1}^\ast {\iota_{m-1,1}}_\ast B \ar[rr]^(.6){\HH^0 \eta} && B
}
\end{equation*}
where $\eta$ is the adjunction map and the right-most vertical map is the identity.  Since $\HH^0 L\iota_{m-1,1}^\ast \bar{\beta}=0$ by assumption, $\eta \circ L\iota_{m-1,1}^\ast \bar{\beta}$ is the zero map since the outer edges of the diagram commute.  However, $\eta \circ L\iota_{m-1,1}^\ast \bar{\beta}$ and $\bar{\beta}$ correspond to each other via the adjunction
\[\Hom_{D(X_k)}(L\iota_{m-1,1}^\ast A, B) \overset{\thicksim}{\to} \Hom_{D(X_{m-1})} (A, {\iota_{m-1,1}}_\ast B).\]
Hence $\bar{\beta}$ itself is zero.
\end{proof}

\begin{lemma}\label{Qmflatness-lemma2}
Let $m \geq 2$.  Then we have the following short exact sequence in $\Ap_m$:
\begin{equation}\label{Qmflatseq1}
 0 \to {\iota_{m,1}}_\ast G \to \tilde{Q}^m \to {\iota_{m,m-1}}_\ast \tilde{Q}^{m-1} \to 0.
\end{equation}
\end{lemma}

\begin{proof}
We begin with the commutative triangle
\begin{equation*}
\xymatrix{
 &  L\iota_m^\ast I^{m-1}\ar[dr] &  \\
L\iota_m^\ast I^m \ar[rr] \ar[ur] & & L\iota_m^\ast I^0
}.
\end{equation*}
Applying the octahedral axiom and taking cohomology with respect to $\Ap_m$, we obtain the commutative diagram
\begin{equation}\label{octahedral6}
\xymatrix{
& & & & & {\iota_{m,1}}_\ast G \ar[ddd] \\
& & {\iota_{m,m-1}}_\ast \tilde{Q}^{m-1} \ar[rrru]^\beta \ar@{^{(}->}[rd]^{\beta_1} & & & \\
& & & L\iota_m^\ast I^{m-1} \ar[dr]^\theta \ar[rruu] & & \\
& \tilde{Q}^m \ar@{^{(}->}[r] \ar[ruu] & L\iota_m^\ast I^m \ar[rr] \ar[ur] & & L\iota_m^\ast I^0 \ar@{->>}[r] \ar@{->>}[rd] & \tilde{Q}^m \ar@{->>}[d] \\
& {\iota_{m,1}}_\ast G \ar@{^{(}->}[ur] \ar@{^{(}->}[u] & & & & {\iota_{m,m-1}}_\ast \tilde{Q}^{m-1}
}
\end{equation}
where the cohomology objects are computed using the flatness of the $I^j$ and Lemma \ref{objectheartLiicohom}.  Define $\theta, \beta, \beta_1$ as above.

Now we move our attention to the commutative triangle
\begin{equation*}
\xymatrix{
  & \image \theta \ar[dr] & \\
L\iota_m^\ast I^{m-1} \ar[rr]^\theta \ar[ur] & & L\iota_m^\ast I^0
}.
\end{equation*}

Applying the octahedral axiom to it gives
\begin{equation*}
\xymatrix{
 & & & {\iota_{m,m-1}}_\ast \tilde{Q}^{m-1} [1] \ar[ddd] \\
& & & \\
& \image \theta \ar[rruu] \ar[dr] & & \\
L\iota_m^\ast I^{m-1} \ar[ur] \ar[rr]^\theta & & L\iota_m^\ast I^0 \ar[r] \ar[rd] & L\iota_m^\ast Q^{m-1} \ar[d] \\
& & & {\iota_{m,m-1}}_\ast \tilde{Q}^{m-1}
}.
\end{equation*}
Note that all the objects in this diagram lie in $D^{\leq 0}_{\Ap_m}(X_m)$.  Pulling back this entire diagram via $L\iota_{m,1}^\ast$, and taking cohomology with respect to $\Ap$, we obtain the following commutative diagram
\begin{equation}\label{octahedral5}
\def\objectstyle{\scriptstyle}
\def\labelstyle{\scriptstyle}
\xymatrix@-1.5pc@R=1pc{
& & & & & 0 \ar[ddd] \\
& & \HH^{-1} L\iota_{m,1}^\ast {\iota_{m,m-1}}_\ast \tilde{Q}^{m-1} \ar[rrru] \ar[dr] & & & \\
& & & \HH^0 L\iota_{m,1}^\ast (\image \theta) \ar[rruu] \ar[dr] & & \\
& G \cong \HH^{-1} L\iota^\ast Q^{m-1} \ar@{^{(}->}[r] \ar[ruu] & L\iota^\ast I^{m-1} \ar[rr]_{\HH^0 L\iota_{m,1}^\ast \theta} \ar[ru] & & L\iota^\ast I^0 \ar@{->>}[r] \ar@{->>}[dr] & \HH^0 L\iota^\ast Q^{m-1} \ar@{->>}[d] \\
& & & & & \HH^0 L\iota_{m,1}^\ast {\iota_{m,m-1}}_\ast \tilde{Q}^{m-1} \\
& G \cong \HH^0 L\iota_{m,1}^\ast {\iota_{m,m-1}}_\ast \tilde{Q}^{m-1} \ar[uu]^\ep \ar[uur]_{\HH^0 L\iota_{m,1}^\ast \beta_1} & & & &  \\
& \HH^{-2} L\iota_{m,1}^\ast {\iota_{m,m-1}}_\ast \tilde{Q}^{m-1} \ar@{^{(}->}[u]^{\ep'} & & & &
}
\end{equation}
where the three straight lines passing through the triangle in the center are exact sequences.  Define $\ep, \ep'$ as above.  Also, the sequence of seven terms that winds around that triangle is the long exact sequence of cohomology for the exact triangle $$L\iota_{m,1}^\ast ( {\iota_{m,m-1}}_\ast \tilde{Q}^{m-1}[1] \to L\iota_m^\ast Q^{m-1} \to {\iota_{m,m-1}}_\ast \tilde{Q}^{m-1} \to {\iota_{m,m-1}}_\ast \tilde{Q}^{m-1}[2] ).$$
Here is why $\HH^0 L\iota_{m,1}^\ast {\iota_{m,m-1}}_\ast \tilde{Q}^{m-1} \cong G$:
\begin{align*}
  \HH^0 L\iota_{m,1}^\ast ({\iota_{m,m-1}}_\ast \tilde{Q}^{m-1}) &\cong \HH^0 L\iota_{m-1,1}^\ast (\HH^0 L\iota_{m,m-1}^\ast {\iota_{m,m-1}}_\ast \tilde{Q}^{m-1}) \text{ by \cite[Corollary 5.4]{Lo1}}\\
  &\cong \HH^0 L\iota_{m-1,1}^\ast \tilde{Q}^{m-1} \text{ by \cite[Lemma 5.6]{Lo1}}\\
  &\cong \HH^0 L\iota_{m-1,1}^\ast (\HH^0 L\iota_{m-1}^\ast Q^{m-1}) \\
  &\cong \HH^0 L\iota^\ast Q^{m-1} \text{ by the right t-exactness of $L\iota_{m-1,1}^\ast$} \\
  &\cong G.
\end{align*}

To understand the other cohomology objects of $\HH^0 L\iota_{m,1}^\ast ({\iota_{m,m-1}}_\ast \tilde{Q}^{m-1})$, we look at
\begin{multline*}
{\iota_{m,1}}_\ast L\iota_{m,1}^\ast {\iota_{m,m-1}}_\ast \tilde{Q}^{m-1} \cong \\
 {\iota_{m,m-1}}_\ast \tilde{Q}^{m-1} \Lo_{\OO_{X_m}} [\cdots \to \OO_{X_m} \overset{\pi}{\to} \OO_{X_m} \overset{\pi^{m-1}}{\to} \OO_{X_m} \overset{\pi}{\to} \OO_{X_m} \to 0].
\end{multline*}
Since multiplication by $\pi^{m-1}$ induces the zero map on objects pushed forward from $D(X_{m-1})$, such as ${\iota_{m,m-1}}_\ast \tilde{Q}^{m-1}$, we have
\begin{equation*}
{\iota_{m,1}}_\ast L\iota_{m,1}^\ast {\iota_{m,m-1}}_\ast \tilde{Q}^{m-1}
   \cong   {\iota_{m,m-1}}_\ast \tilde{Q}^{m-1} \Lo_{\OO_{X_m}}  \left( \oplus_{i \geq 0} [\OO_{X_m} \overset{\pi}{\to} \OO_{X_m}][2i] \right).
\end{equation*}
Therefore
\begin{equation*}
  \HH^i L\iota_{m,1}^\ast {\iota_{m,m-1}}_\ast \tilde{Q}^{m-1} \text{ are isomorphic for all even $i\leq 0$}
\end{equation*}
and in particular are isomorphic to $G$.  As a consequence,  in figure \eqref{octahedral5}, the map $\ep'$ must be an isomorphism.  So $\ep$ must be the zero map, implying $\HH^0 L\iota_{m,1}^\ast \beta_1$ is also zero.  Thus $\HH^0 L\iota_{m,1}^\ast \beta$ itself is zero from the commutativity of diagram \eqref{octahedral6}.

Now, by \cite[Lemma 5.7]{Lo1}, $\beta = {\iota_{m,m-1}}_\ast \bar{\beta}$ for some $\bar{\beta} : \tilde{Q}^{m-1} \to {\iota_{m-1,1}}_\ast G$.  This gives us
\begin{align*}
  0 &= \HH^0 L\iota_{m,1}^\ast \beta \\
  &= \HH^0 L\iota_{m,1}^\ast {\iota_{m,m-1}}_\ast \bar{\beta} \\
  &\cong \HH^0 L\iota_{m-1,1}^\ast (\HH^0 L\iota_{m,m-1}^\ast {\iota_{m,m-1}}_\ast \bar{\beta}) \text{ by \cite[Corollary 5.4]{Lo1}}\\
  &\cong \HH^0 L\iota_{m-1,1}^\ast \bar{\beta} \text{ by \cite[Lemma 5.6]{Lo1}}.
\end{align*}
Using Lemma \ref{Qmflatness-lemma1}, we conclude $\bar{\beta}=0$ and thus $\beta =0$.  This means that, finally, the sequence of outer six terms in Figure \eqref{octahedral6} breaks up into two short exact sequences, one of them being the one we want.
\end{proof}

%\noindent\textit{Remark.} Applying the octahedral axiom twice in a row in the above proof seems rather inefficient.

Before we move on, let us define the objects $R^{m,m-r}$ and $Q^{m,r}$.  These definitions will help clarify our arguments.

\noindent\textit{The objects $R^{m,m-r}$.}  Since we have an exact sequence as in Lemma \ref{Qmflatness-lemma2} for each $m \geq 2$, we have a sequence of surjections
$$ \tilde{Q}^m \twoheadrightarrow {\iota_{m,m-1}}_\ast \tilde{Q}^{m-1} \twoheadrightarrow {\iota_{m,m-2}}_\ast \tilde{Q}^{m-2} \twoheadrightarrow \cdots \twoheadrightarrow {\iota_{m,m-r}}_\ast \tilde{Q}^{m-r}$$
in $\Ap_m$ for $m-r \geq 1$.  Define $R^{m,m-r} \in \Ap_m$ to be the kernel of this composition of surjections.

On the other hand, if we apply the octahedral axiom to the commutative triangle
\begin{equation*}
  \xymatrix{
  & {\iota_{m,m-1}}_\ast \tilde{Q}^{m-1} \arsurj[dr] &\\
  \tilde{Q}^m \arsurj[ur] \arsurj[rr] & & {\iota_{m,m-2}}_\ast \tilde{Q}^{m-2}
  },
\end{equation*}
we get
\begin{equation*}
\xymatrix{
 & & & {\iota_{m,1}}_\ast G [1] \ar[ddd]\\
& & & \\
& {\iota_{m,m-1}}_\ast \tilde{Q}^{m-1} \ar[rruu] \ar[dr] & & \\
\tilde{Q}^m \arsurj[ur] \arsurj[rr] & & {\iota_{m,m-2}}_\ast \tilde{Q}^{m-2} \ar[r] \ar[dr] & R^{m,m-2}[1] \ar[d] \\
& & & {\iota_{m,1}}_\ast G[1]
}.
\end{equation*}
That is, we have a short exact sequence
\begin{equation*}
0 \to {\iota_{m,1}}_\ast G \to R^{m,m-2} \to {\iota_{m,1}}_\ast G \to 0 \text{\quad in $\Ap_m$}.
\end{equation*}
Iterating this process (by looking at the output of the octahedral axiom applied to the composition $\tilde{Q}^m \to {\iota_{m,m-(r-1)}}_\ast \tilde{Q}^{m-(r-1)} \to {\iota_{m,m-r}}_\ast \tilde{Q}^{m-r}$), we obtain short exact sequences of the form
\begin{equation}\label{GQRsequence}
0 \to R^{m,m-(r-1)} \to R^{m,m-r} \to {\iota_{m,1}}_\ast G \to 0
\end{equation}
in $\Ap_m$ for each $m > r \geq 2$.  Comparing these with the short exact sequence of Lemma \ref{Qmflatness-lemma2}, we see that $R^{m,m-r}$ and ${\iota_{m,r}}_\ast \tilde{Q}^r$ have the same Hilbert polynomial, both being $r$ times the Hilbert polynomial of ${\iota_{m,1}}_\ast G$.

\noindent\textit{The Objects $Q^{m,r}$.}  For any $0\leq r<m$, we can define the objects $Q^{m,r}$ by the exact triangle in $D^b(X_R)$
\begin{equation}\label{definition-Qmr}
  I^m \to I^r \to Q^{m,r} \to I^m[1]
\end{equation}
where the morphism $I^m \to I^r$ comes from composition of the elementary modifications.  (So $Q^{m,0}$ is simply $Q^m$.)

The octahedral axiom gives us diagrams such as
\begin{equation*}
\xymatrix{
 & & & \iota_\ast G  \ar[ddd]\\
& & & \\
& I^{m-1} \ar[rruu] \ar[dr] & & \\
I^m \ar[ur] \ar[rr] & & I^{m-2} \ar[r] \ar[dr] & Q^{m,m-2} \ar[d] \\
& & & \iota_\ast G
}
\end{equation*}
from which we see $Q^{m,m-2} = {\iota_2}_\ast \tilde{Q}^{m,m-2}$ for some $\tilde{Q}^{m,m-2} \in D^b(X_2)$ (Lemma \ref{XmXnXm+n}).

If we fix $m$ and iterate this process, we would obtain $Q^{m,r} = {\iota_{m-r}}_\ast \tilde{Q}^{m,r}$ for some $\tilde{Q}^{m,r} \in D^b(X_{m-r})$.  Since the heart of a t-structure is extension-closed, each $\tilde{Q}^{m,r}$ lies in the heart $\Ap_{m-r}$.  It is easy to see that $\tilde{Q}^{m-r}$ and $\tilde{Q}^{m,r}$ have the same Hilbert polynomial.

\begin{lemma}\label{Q2mflat-lemma1}
Let $m \geq 2$.  If $\tilde{Q}^m$ is flat over $\Spec R/\pi^m$ then $\HH^0 L\iota_{2m,m}^\ast \tilde{Q}^{2m} \cong \tilde{Q}^m$.
\end{lemma}
\begin{proof}
Applying the octahedral axiom to the commutative triangle
\[
\xymatrix{
& I^m\ar[dr] & \\
I^{2m} \ar[ur] \ar[rr] & & I^0
}
\]
gives us an exact triangle
\[
  Q^{2m,m} \to Q^{2m} \overset{\al}{\to} Q^m \to Q^{2m,m}[1].
\]
Applying $\HH^0 L\iota_m^\ast$, we get a short exact sequence in $\Ap_m$
\begin{equation*}
\xymatrix{
0 \ar[r] &  K \ar[r] & \HH^0 L\iota_m^\ast Q^{2m} \ar[rr]^{\HH^0 L\iota_m^\ast \al}& & \HH^0 L\iota_m^\ast Q^m \cong \tilde{Q}^m \ar[r] & 0
}
\end{equation*}
where $K$ is the kernel of $\HH^0 L\iota_m^\ast \al$.  Using the flatness of $\tilde{Q}^m$, we can pull back this sequence to an exact triangle via $L\iota_{m,1}^\ast$ and take cohomology with respect to $\Ap$.  Since $\HH^0 L\iota_{m,1}^\ast (\HH^0 L\iota_m^\ast Q^{2m})$ and $\HH^0 L\iota_{m,1}^\ast (\HH^0 L\iota_m^\ast Q^m)$ are both isomorphic to $G$, the morphism $\HH^0 L\iota_{m,1}^\ast (\HH^0 L\iota_m^\ast \al)$ is an isomorphism.  And because $\HH^{-1} L\iota_{m,1}^\ast \tilde{Q}^m = 0$, we have $\HH^0 L\iota_{m,1}^\ast K=0$.  Hence $K=0$ by \cite[Lemma 5.5]{Lo1}.  As a result,
\[
\tilde{Q}^m \cong \HH^0 L\iota_m^\ast Q^{2m} \cong \HH^0 L\iota_{2m,m}^\ast L\iota_{2m}^\ast {\iota_{2m}}_\ast \tilde{Q}^{2m} \cong \HH^0 L\iota_{2m,m}^\ast \tilde{Q}^{2m}
\]
 as wanted.
\end{proof}

\begin{pro}\label{QmflatQ2mflat}
For $m \geq 1$, if $\tilde{Q}^m$ is flat, then $\tilde{Q}^{2m}$ is also flat.
\end{pro}

\begin{proof}
We need to demonstrate that  the cohomology of $L\iota_{2m,1}^\ast \tilde{Q}^{2m}$ only exists at degree 0 with respect to the heart $\Ap$.

Consider the double complex in $\Ap_{2m}$
\begin{equation}\label{Q2mflat-doublecomplex}
\xymatrix{
0 \ar[r] & R^{2m,m} \ar[r]^{\tilde{\phi}} & \tilde{Q}^{2m} \ar[r]^{\tilde{\psi}} & {\iota_{2m,m}}_\ast \tilde{Q}^{m} \ar[r] & 0 \\
0 \ar[r] & R^{2m,m} \ar[r]^{\tilde{\phi}} \ar[u]^{\pi^m} & \tilde{Q}^{2m} \ar[r]^{\tilde{\psi}}  \ar[u]^{-\pi^m} & {\iota_{2m,m}}_\ast \tilde{Q}^{m} \ar[r] \ar[u]^{\pi^m} & 0 \\
0 \ar[r] & R^{2m,m} \ar[r]^{\tilde{\phi}} \ar[u]^{\pi^m} & \tilde{Q}^{2m} \ar[r]^{\tilde{\psi}}  \ar[u]^{-\pi^m} & {\iota_{2m,m}}_\ast \tilde{Q}^{m} \ar[r] \ar[u]^{\pi^m} & 0  \\
 & \vdots \ar[u]^{\pi^m} & \vdots \ar[u]^{-\pi^m} & \vdots \ar[u]^{\pi^m} &
}
\end{equation}
where the rows are  copies of the short exact sequence $0 \to R^{2m,m} \to \tilde{Q}^{2m} \to {\iota_{2m,m}}_\ast \tilde{Q}^m \to $ that comes from the definition of $R^{2m,m}$, and the vertical maps are all multiplication by $\pm \pi^m$.  The diagram \eqref{Q2mflat-doublecomplex}  has two  properties:
\begin{itemize}
\item[(a)] Each column is a chain complex in the abelian category $\Ap_{2m}$ (i.e.\ the composition of two successive differentials is zero).  In particular, the differentials in the first and third columns are all zero maps.
\item[(b)] Consider each column of \eqref{Q2mflat-doublecomplex} as a complex in $D(\Ap_{2m})$ sitting at degrees $\leq 0$.  Then the degree-$i$ cohomology of the first, second the third column can be computed by applying $\HH^i_{\Ap_{2m}} ({\iota_{2m,m}}_\ast L\iota_{2m,m}^\ast (-))$ to $R^{2m,m}, \tilde{Q}^{2m}$ and ${\iota_{2m,m}}_\ast \tilde{Q}^m$, respectively.
\end{itemize}

From the construction of $R^{2m,m}$, we know that $R^{2m,m}$ is the pushforward of an object in $\Ap_m$, i.e.\ $R^{2m,m} \cong {\iota_{2m,m}}_\ast \tilde{R}^{2m,m}$ for some $\tilde{R}^{2m,m}$ (using Lemma \ref{XmXnXm+n}).  It is then clear that property (a) holds.  To see property (b), we can use Postnikov systems (see \cite[Section 1.3]{Orlov}, for example).  Consider the following Postnikov system in a triangulated category $\mathcal{D}$
\begin{equation*}
\def\objectstyle{\scriptstyle}
\def\labelstyle{\scriptstyle}
\xymatrix@-.5pc@R=2pc{
X^{-3}  \ar@{}[drr]^{\circ} \ar[rr]^{d^{-3}} \ar[dr]_{u_{-3}=\mbox{id}} &&   X^{-2} \ar@{}[drr]^{\circ} \ar[rr]^{d^{-2}} \ar[dr]_{u_{-2}} &  & X^{-1}  \ar@{}[drr]^\circ \ar[dr]_{u_{-1}} \ar[rr]^{d^{-1}}  &  & X^{0} \ar[dr]_{u_{0}} & \\
& Y^{-3} = X^{-3} \ar@{}[urr]^\ast  \ar[ur]_{v_{-3}}  && Y^{-2}\ar@{}[urr]^\ast \ar[ur]_{v_{-2}}  \ar[ll]^{[1]} & & Y^{-1} \ar@{}[urr]^\ast \ar[ur]_{v_{-1}} \ar[ll]^{[1]} & & Y^{0} \ar[ll]^{[1]}
}
\end{equation*}
where triangles marked with $\circ$ are commutative, the triangles marked with $\ast$ are  exact triangles, and $X^\bullet = [X^{-3} \overset{d^{-3}}{\to}  X^{-2} \overset{d^{-2}}{\to} X^{-1} \overset{d^{-1}}{\to}  X^0]$ is a chain complex in the category $\mathcal{D}$ (i.e.\ $d^{i}d^{i+1}=0$ for any $i$).  Given such a Postnikov system in $\mathcal{D}$, if we also have a t-structure on $\mathcal{D}$ with heart $\mathcal{A}$, and $X^\bullet$ is a complex with all the terms in $\mathcal{A}$, then
\begin{equation}\label{cohom-postnikov}
H^i(X^\bullet) \cong \HH^i_{\mathcal{A}} (Y^0)
\end{equation}
where $H^i(X^\bullet)$ is the degree-$i$ cohomology of $X^\bullet$ considered as a chain complex in $\mathcal{A}$, and $\HH^i_{\mathcal{A}} (Y^0)$ is the degree-$i$ cohomology of $Y^0$ computed in $\mathcal{D}$ with respect to the t-structure with heart $\mathcal{A}$.  (To see this,  apply cohomology functors $\HH^i_{\mathcal{A}}$ to the Postnikov system.)

For our situation, we build a Postnikov system in $D(X_{2m})$ of the above form  as follows: let $F^\bullet = [F^{-1} \overset{d}{\to} F^0] $ be any 2-term complex of coherent sheaves representing an object of $\Ap_{2m}$.  That is,  $F^{-1}, F^0$ are coherent sheaves on $X_{2m}$, and $d$ is a morphism of coherent sheaves.  At some point, we will specify $F^\bullet$ to be $R^{2m,m}, \tilde{Q}^{2m}$ or ${\iota_{2m,m}}_\ast \tilde{Q}^m$; however, we will not do so just yet.  We can consider $F^\bullet$  as an object in the derived category $D(\Ap_{2m})$ via the canonical inclusion $\Ap_{2m} \to D(\Ap_{2m})$.  Note that $D(\Ap_{2m})=D(\Coh(X_{2m}))$ by \cite[Proposition 5.4.3]{BV}, although we will not explicitly use this fact  in this proof.

For $-3\leq i \leq 0$, let $X^i = F^\bullet$, and for $-3 \leq i \leq -1$, let  $d^i$ be multiplication by $\pi^m$.  Then $X^\bullet =[X^{-3} \overset{d^{-3}}{\to} X^{-2} \overset{d^{-2}}{\to} X^{-1} \overset{d^{-1}}{\to}  X^0]$ is a chain complex in the abelian category $\Ap_{2m}$.  Let $v_{-3}$ be multiplication by $\pi^m$.  For $-2 \leq i \leq 0$, let $Y^{i}$ be the cone of $v_{i-1}$, and let $u_{i}$ be the natural inclusion of $X^{i}$ into $Y^{i}$.  Also, for $i=-2,-1$, let $v_{i}$ be the projection of $Y^{i}$ onto $X^{i}=X^{i+1}$ followed by multiplication by $\pi^m$.  Note that, while the natural projection of $Y^{i}$ onto $X^{i}$ itself is \textit{not} a chain map in $\text{Kom}(X_{2m})$, it becomes a chain map when we post-compose it with multiplication by $\pi^m$.  This construction gives us a Postnikov system with
\begin{equation*}
  Y^0 = \cone (v_{-1}) = \begin{matrix}
\xymatrix@-1pc@R=1.5pc{ {\save
[].[dddr]*[F.]\frm{}\restore}
  F^{-1} \ar[r]^{d} & F^0 \\
  F^{-1} \ar[u]^{\pi^m} \ar[r]^{-d} & F^0 \ar[u]^{\pi^m} \\
  F^{-1} \ar[u]^{-\pi^m} \ar[r]^d & F^0 \ar[u]^{-\pi^m} \\
  F^{-1} \ar[u]^{\pi^m} \ar[r]^{-d} & F^0 \ar[u]^{\pi^m}
  }
  \end{matrix}
\end{equation*}
where the dotted box means to `take the total complex of the double complex inside,' a notation we adopt until the end of the proof of this proposition.  Also, here we use the sign convention used in \cite[Definition 2.15]{FMTAG} in the definition for the mapping cone.

Now, for the infinite complex $[ \cdots F^\bullet \overset{\pi^m}{\to} F^\bullet \overset{\pi^m}{\to} F^\bullet ]$ in $D(\Ap_{2m})$, the cohomology at all degrees $i \leq -1$ are isomorphic by periodicity.  On the other hand, for any $i \leq -1$ we have
\begin{align*}
  H^i  ( [ \cdots \overset{\pi^m}{\to} F^\bullet \overset{\pi^m}{\to} F^\bullet \overset{\pi^m}{\to} F^\bullet ])
  & \cong H^{-1} ([ F^\bullet \overset{\pi^m}{\to} F^\bullet \overset{\pi^m}{\to} F^\bullet \overset{\pi^m}{\to} F^\bullet ]) \\
  &= \HH^{-1}_{\Ap_{2m}} (Y^0) \text{ by our Postnikov system above and \eqref{cohom-postnikov}} \\
  &= \HH^{-1}_{\Ap_{2m}} \begin{pmatrix}
 \def\g#1{\save
[].[dddr]*[F.]\frm{}\restore}%
\xymatrix@-1pc@R=1.5pc{ \g1
  F^{-1} \ar[r]^{d} & F^0 \\
  F^{-1} \ar[u]^{\pi^m} \ar[r]^{-d} & F^0 \ar[u]^{\pi^m} \\
  F^{-1} \ar[u]^{-\pi^m} \ar[r]^d & F^0 \ar[u]^{-\pi^m} \\
  F^{-1} \ar[u]^{\pi^m} \ar[r]^{-d} & F^0 \ar[u]^{\pi^m}
  }
  \end{pmatrix} \\
  &\cong \HH^{-1}_{\Ap_{2m}} \begin{pmatrix} \def\g#1{\save
[].[dddr]*[F.]\frm{}\restore}%
\xymatrix@-1pc@R=1.5pc{ \g1
  F^{-1} \ar[r]^{d} & F^0 \\
  F^{-1} \ar[u]^{-\pi^m} \ar[r]^d & F^0 \ar[u]^{\pi^m} \\
  F^{-1} \ar[u]^{-\pi^m} \ar[r]^d & F^0 \ar[u]^{\pi^m} \\
  F^{-1} \ar[u]^{-\pi^m} \ar[r]^{d} & F^0 \ar[u]^{\pi^m}
  }
  \end{pmatrix}
\end{align*}
since we have a quasi-isomorphism between these two complexes, given by the chain map
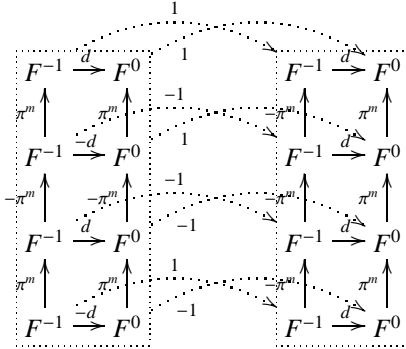
\begin{equation*}
 \def\g#1{\save
[].[dddr]*[F.]\frm{}\restore}%
\xymatrix@-1pc@R=1.5pc{ \g1
  F^{-1} \ar[r]^{d} & F^0 &&& \g2 F^{-1} \ar[r]^{d} & F^0 \\
  F^{-1} \ar[u]^{\pi^m} \ar[r]^{-d} & F^0 \ar[u]^{\pi^m} &&&  F^{-1} \ar[u]^{-\pi^m} \ar[r]^d & F^0 \ar[u]^{\pi^m} \\
  F^{-1} \ar[u]^{-\pi^m} \ar[r]^d & F^0 \ar[u]^{-\pi^m} &&& F^{-1} \ar[u]^{-\pi^m} \ar[r]^d & F^0 \ar[u]^{\pi^m} \\
  F^{-1} \ar[u]^{\pi^m} \ar[r]^{-d} & F^0 \ar[u]^{\pi^m} &&&  F^{-1} \ar[u]^{-\pi^m} \ar[r]^{d} & F^0 \ar[u]^{\pi^m}
  \ar@/^1.5pc/ @{.>}^1 "1,1" ; "1,5"
  \ar@/^1.5pc/ @{.>}_(.2)1 "1,2" ; "1,6"
  \ar@/^1.5pc/ @{.>}^{-1} "2,1" ; "2,5"
  \ar@/^1.5pc/ @{.>}_(.2)1 "2,2" ; "2,6"
  \ar@/^1.5pc/ @{.>}^{-1} "3,1" ; "3,5"
  \ar@/^1.5pc/ @{.>}_(.2){-1} "3,2" ; "3,6"
  \ar@/^1.5pc/ @{.>}^1 "4,1" ; "4,5"
  \ar@/^1.5pc/ @{.>}_(.2){-1} "4,2" ; "4,6"
  }
  \end{equation*}
Then,
\begin{equation*}
\HH^{-1}_{\Ap_{2m}} \begin{pmatrix} \def\g#1{\save
[].[dddr]*[F.]\frm{}\restore}%
\xymatrix@-1pc@R=1.5pc{ \g1
  F^{-1} \ar[r]^{d} & F^0 \\
  F^{-1} \ar[u]^{-\pi^m} \ar[r]^d & F^0 \ar[u]^{\pi^m} \\
  F^{-1} \ar[u]^{-\pi^m} \ar[r]^d & F^0 \ar[u]^{\pi^m} \\
  F^{-1} \ar[u]^{-\pi^m} \ar[r]^{d} & F^0 \ar[u]^{\pi^m}
  }
  \end{pmatrix} \cong \HH^{-1}_{\Ap_{2m}} \begin{pmatrix}
  \def\g#1{\save
[].[ddddr]*[F.]\frm{}\restore}%
\xymatrix@-1pc@R=1.5pc{ \g1
  F^{-1} \ar[r]^{d} & F^0 \\
  F^{-1} \ar[u]^{-\pi^m} \ar[r]^{d} & F^0 \ar[u]^{\pi^m} \\
  F^{-1} \ar[u]^{-\pi^m} \ar[r]^d & F^0 \ar[u]^{\pi^m} \\
  F^{-1} \ar[u]^{-\pi^m} \ar[r]^{d} & F^0 \ar[u]^{\pi^m} \\
  \vdots \ar[u] & \vdots \ar[u]
  }
  \end{pmatrix}
  \cong  \HH^i_{\Ap_{2m}} ({\iota_{2m,m}}_\ast L\iota_{2m,m}^\ast F^\bullet)
\end{equation*}
where the first isomorphism follows because we only need the top four rows in the middle dotted box to compute the cohomology $\HH^{-1}(\Ap_{2m})$. And so, overall, we obtain
\[
 H^i ( [ \cdots \overset{\pi^m}{\to} F^\bullet \overset{\pi^m}{\to} F^\bullet \overset{\pi^m}{\to} F^\bullet ]) \cong \HH^i_{\Ap_{2m}} ({\iota_{2m,m}}_\ast L\iota_{2m,m}^\ast F^\bullet).
\]
Taking $F^\bullet$ to be $R^{2m,m}, \tilde{Q}^{2m}$ or ${\iota_{2m,m}}_\ast \tilde{Q}^m$, we obtain  property (b).

Now, by properties (a) and (b), the long exact sequence of the double complex \eqref{Q2mflat-doublecomplex} in $\Ap_{2m}$ looks like
\begin{equation*}
 \xymatrix{
 {\iota_{2m,m}}_\ast \tilde{R}^{2m,m} \ar[r]^(.4){\phi_0} & {\iota_{2m,m}}_\ast \HH^0 L\iota_{2m,m}^\ast \tilde{Q}^{2m} \ar[r]^(.6){\psi_0} & {\iota_{2m,m}}_\ast \tilde{Q}^m \ar[r] & 0 \\
 {\iota_{2m,m}}_\ast \tilde{R}^{2m,m} \ar[r]^(.4){\phi_1} & {\iota_{2m,m}}_\ast \HH^{-1} L\iota_{2m,m}^\ast \tilde{Q}^{2m} \ar[r]^(.6){\psi_1} & {\iota_{2m,m}}_\ast \tilde{Q}^m \ar `u[ll] `[ull]_{\delta_1} [ull]    & \\
 {\iota_{2m,m}}_\ast \tilde{R}^{2m,m} \ar[r]^(.4){\phi_2} & {\iota_{2m,m}}_\ast \HH^{-2} L\iota_{2m,m}^\ast \tilde{Q}^{2m} \ar[r]^(.6){\psi_2} & {\iota_{2m,m}}_\ast \tilde{Q}^{m} \ar `u[ll] `[ull]_{\delta_2} [ull]    &\\
& \vdots & \ar `u[ll] `[ull] [ull]    &
}.
\end{equation*}

From Lemma \ref{Q2mflat-lemma1}, we know $\psi_0$ is an isomorphism.  So $\phi_0$ is the zero map, and $\delta_1$ is an isomorphism (since ${\iota_{2m,m}}_\ast \tilde{R}^{2m,m}$ and ${\iota_{2m,m}}_\ast \tilde{Q}^m$ have the same Hilbert polynomial).  Consequently, $\psi_1$ is the zero map.  However, the double complex \eqref{Q2mflat-doublecomplex} is 1-periodic in the rows, and so all the $\psi_i$ are zero maps for $i \geq 1$.  That $\psi_2$ is zero means $\delta_2$ is an isomorphism (by comparing Hilbert polynomials), hence $\phi_1$ is the zero map.  By the same periodicity argument, we get that all $\phi_i$ are zero maps for $i \geq 1$.  As a result, we obtain $\HH^i L\iota_{2m,m}^\ast \tilde{Q}^{2m}=0$ for all $i\leq -1$, thus $L\iota_{2m,m}^\ast \tilde{Q}^{2m} \cong \tilde{Q}^m$.  Since $\tilde{Q}^m$ is already flat, we see that $\tilde{Q}^{2m}$ is also flat.
\end{proof}

\begin{lemma}\label{flatness-2mimplies2m-1}
For $m\geq 3$, if $\tilde{Q}^{m}$ is flat, then $\tilde{Q}^{m-1}$ is also flat.
\end{lemma}

\begin{proof}

By Lemma \ref{Qmflatness-lemma2}, we have the following short exact sequence in $\Ap_m$:
\begin{equation*}
 0 \to {\iota_{m,1}}_\ast G \to \tilde{Q}^m \to {\iota_{m,m-1}}_\ast \tilde{Q}^{m-1} \to 0.
\end{equation*}

The flatness of $\tilde{Q}^m$ implies $L\iota^\ast \tilde{Q}^m \in \Ap (X_k)$, which in turn implies $L\iota_{m,m-1}^\ast \tilde{Q}^m \in \Ap_{m-1}$ by \cite[Proposition 5.10]{Lo1}.  Applying $L\iota_{m,m-1}^\ast$ and taking the long exact sequence of cohomology with respect to $\Ap_{m-1}$, we get
\[
\xymatrix{
 {\iota_{m-1,1}}_\ast G \ar[r] & L\iota_{m,m-1}^\ast \tilde{Q}^{m} \ar[r] & \tilde{Q}^{m-1} \ar[r] & 0 \\
{\iota_{m-1,1}}_\ast G \ar[r] & 0 \ar[r] & \HH^{-1} (L\iota_{m,m-1}^\ast {\iota_{m,m-1}}_\ast \tilde{Q}^{m-1}) \ar `u[ll] `[ull]_\delta [ull] & \\
\cdots \ar[r] & 0 \ar[r] & \HH^{-2} (L\iota_{m,m-1}^\ast {\iota_{m,m-1}}_\ast \tilde{Q}^{m-1})\ar `u[ll] `[ull] [ull]  &
}.
\]
By \cite[Lemma 5.9]{Lo1}, for all odd integers $i<0$, the  objects $\HH^i L\iota_{m,m-1}^\ast {\iota_{m,m-1}}_\ast \tilde{Q}^{m-1}$ are isomorphic.  In particular, the cohomology $\HH^{-1} L\iota_{m,m-1}^\ast {\iota_{m,m-1}}_\ast \tilde{Q}^{m-1}$ is isomorphic to the cohomology $\HH^{-3} L\iota_{m,m-1}^\ast {\iota_{m,m-1}}_\ast \tilde{Q}^{m-1}$, which is isomorphic to ${\iota_{m-1,1}}_\ast G$.  This reveals that $\delta$ is an isomorphism, and so $L\iota_{m,m-1}^\ast \tilde{Q}^m \cong \tilde{Q}^{m-1}$.  However, that $\tilde{Q}^m$ is flat means that $L\iota_{m,m-1}^\ast \tilde{Q}^m$ is also flat.  Therefore, $\tilde{Q}^{m-1}$ is also flat.
\end{proof}

This concludes the proof of

\begin{pro}
$\tilde{Q}^m$ is flat over $\Spec R/\pi^m$ for each $m\geq 1$.
\end{pro}

\subsubsection{Lifting the Inverse System to a Family over $\Spec \hat{R}$}

Now that we have shown flatness for each $\tilde{Q}^m$, we know that the surjections $L\iota_m^\ast I^0 \overset{\al_m}{\longrightarrow} \HH^0 L\iota_m^\ast Q^m$ form an inverse system with respect to the functors $\HH^0_{\Ap_m} \circ L\iota_{m,m'}^\ast : \Ap_m \to \Ap_{m'}$, as described in section \ref{sssec-invsysconst}.

Now, fix an isomorphism $c_m : L\iota_{m+1,m}^\ast \tilde{Q}^{m+1} \cong \tilde{Q}^{m}$ for each $m\geq 1$.  Then we can use these isomorphisms and their compositions with the various $L\iota_{m,m'}^\ast c_m$ to construct compatible isomorphisms that, together with the $\{\tilde{Q}^m\}_m$, give  an inverse system that satisfy the hypotheses of:

\begin{pro}\cite[Prop 3.6.1]{Lieblich}
Let $X \to S$ be a flat morphism of finite presentation of quasi-separated algebraic spaces.  Let $(A,\mm,k)$ be a complete local Noetherian $S$-ring and $E_i \in D^b(X_{A/\mm^{i+1}})$ a system of elements with compatible isomorphisms $E_i \Lo_{A/\mm^{i+1}} A/\mm^{i} \overset{\thicksim}{\to} E_{i-1}$ in the derived category.  Then there is $E \in D^b(X_A)$ and compatible isomorphisms $E \Lo_A A/\mm^{i+1} \overset{\thicksim}{\to} E_i$.
\end{pro}

As a result, we have

\begin{pro}
If $R$ is a complete DVR over $k$, then there exists an object $Q_R
 \in D^b(X_R)$ and compatible isomorphisms $L\iota_m^\ast Q_R \cong \tilde{Q}^m$.
\end{pro}

By concatenating the commutative triangles $\xymatrix{ I^0 \ar[r] \ar[dr] & Q^m \ar[d] \\ & Q^{m-1}}$ extracted from diagram \eqref{octahedral1} for $m \geq 1$, and fixing an isomorphism $Q^m \overset{\thicksim}{\to} {\iota_m}_\ast \iota_m^\ast Q_R$ for each $m$, we get a commutative diagram
\begin{equation}
\xymatrix{
  I^0 \ar[r] \ar[dr] \ar[ddr] & {\iota_m}_\ast \iota_m^\ast Q_R \ar[d]\\
   & {\iota_{m-1}}_\ast \iota_{m-1}^\ast Q_R \ar[d] \\
   & {\iota_{m-2}}_\ast \iota_{m-2}^\ast Q_R \ar[d] \\
    & \vdots
}.
\end{equation}
Now, by the following lemma, we have an isomorphism
\begin{equation*}
  \Hom_{X_R} (I^0, Q_R) = \varprojlim \Hom_{X_R} (I^0, R/\pi^{i+1} \Lo Q_R).
\end{equation*}
Note that ${\iota_m}_\ast \iota_m^\ast Q_R \cong R/\pi^m \Lo Q_R$.

\begin{lemma}\cite[Lemma 4.1.1(2)]{Lieblich}
Let $X \to \Spec A_0$ be a proper flat algebraic space of finite presentation over a reduced Noetherian ring and $I$ a finite $A_0$-module.  Given $E,F \in D^b_p (X/A_0)$ and any $i$, if $\mm \subset A_0$ is a maximal ideal, then
\[
  \Ext^i_X (E, F \Lo I) \otimes_{A_0} \hat{A}_0 = \varprojlim \Ext^i_X (E,I/\mm^{i+1} I \Lo F).
\]
\end{lemma}
Even though the lemma is stated only for the case $E=F$ in \cite{Lieblich}, the proof works in general.

Therefore, the commutative diagram above lifts to an element  $\varphi \in \Hom_{X_R} (I^0,Q_R)$.  %(Lieblich said that the original proof for \cite{Lemma 4.1.1(2)}[MCPM] still works even even if the '$E$' in the first and second arguments of $\Ext^i_X$ are different.).
And $\varphi$ restricts to the maps $\al_m : L\iota^\ast_m I^0 \twoheadrightarrow \tilde{Q}^m$ (see the start of section \ref{sssec-invsysconst}).

\subsection{Universal Closedness over an Arbitrary DVR}

The goal of this section is to show the  valuative criterion of universal closedness for PT-semistable objects.

So far, we have learned that if $R$ is a complete DVR, and $E_K \in \Ap_K$ is PT-semistable with $ch_0(E_K) \neq 0$, then we can produce an $R$-flat object $I^0 \in D^b(X_R)$ that restricts to $E_K$ on $X_K$.  We have also learned, that if the central fibre $L\iota^\ast I^i$ does not become PT-semistable after a finite sequence of elementary modifications, we can  produce a morphism $\varphi : I^0 \to Q_R$ where $Q_R$ is $R$-flat, and $L\iota^\ast (\varphi)$ is a destabilising quotient $L\iota^\ast I^0 \to L\iota^\ast Q_R \cong G$ in $\Ap$.  In this case, we can define $A \in D^b(X_R)$ by the exact triangle $A \to I^0 \overset{\varphi}{\to} Q_R \to A[1]$.  Pulling back this exact triangle via $j^\ast$ gives
\begin{equation}\label{jstarI0QR}
 j^\ast A \to j^\ast I^0 \overset{j^\ast\varphi}{\longrightarrow} j^\ast Q_R \to j^\ast A[1].
 \end{equation}
On the other hand, we can also pull back the same triangle using $L\iota^\ast$ to obtain
$$ L\iota^\ast A \to L\iota^\ast I^0 \overset{L\iota^\ast\varphi}{\longrightarrow} L\iota^\ast Q_R \to L\iota^\ast A[1].$$
Since $L\iota^\ast \varphi$ is a surjection in $\Ap$, $L\iota^\ast A$ is necessarily in $\Ap$.  So by openness of the heart $\Ap$ \cite[Lemma 3.14]{TodaLSOp}, $j^\ast A$ and $j^\ast Q_R$ both lie in the heart $\Ap_K$.  Since Chern character is locally constant, $L\iota^\ast I^0$ and $j^\ast I^0$ would have the same Chern character, as are $L\iota^\ast Q_R$ and $j^\ast Q_R$.  This means that  \eqref{jstarI0QR} yields a short exact sequence in $\Ap_K$ that destabilises $j^\ast I^0 \cong E_K$, which is a contradiction.  We have just proved:

\begin{theorem}[Valuative criterion for universal closedness]\label{theorem-prelimcompletenessofPTss}
Let $(R,\pi)$ be a DVR, and $K$ its field of fractions.  Given a PT-semistable object $E_K \in \Ap (X_K)$ such that $ch_0(E_K)\neq 0$, there exists $E \in D^b (X_R)$, a flat family of objects in $\Ap$ over $\Spec R$ such that $j^\ast E \cong E_K$ and $L\iota^\ast E$  is PT-semistable.
\end{theorem}

\begin{proof}
Our proof above is for the case where $R$ is a complete DVR.  This is sufficient for us to define the moduli of PT-semistable objects and conclude that it is an Artin stack of finite type over $k$ (see Proposition \ref{pro-openness} and Section \ref{section-modulidef}).  Then, by \cite[Remark 4, (7.4)]{LMB} and \cite[Theorem (7.10)]{LMB}, the result for an arbitrary DVR follows.
\end{proof}

%[Have I mentioned that for a DVR $(R,\pi,k)$, its field of fractions is $R[\frac{1}{\pi}]$?]

If we did not wish to use the existence of our moduli space as an Artin stack in showing the valuative criterion for an arbitrary DVR, we could prove the above theorem under an additional hypothesis, as follows:

Consider the commutative diagram where $R$ is an arbitrary DVR and $\hat{R}$ is its completion, and $K$ and $\hat{K}$ are their respective fields of fractions:
\begin{equation*}
\xymatrix{
  X_k \arinj[r]^\iota \ar[d] & X_{\hat{R}} \ar[d]^p & X_{\hat{K}} \ar@{_{(}->}[l]_j \ar[d]^p \\
  X_k \arinj[r]^\iota & X_R & X_K \ar@{_{(}->}[l]_j
}.
\end{equation*}
Suppose we are given a PT-semistable object $E_K$ in the heart $\Ap (X_K)$.  The additional hypothesis we need to add is, that $ch_0 (E_K) \neq 0$ and $ch_0(E_K), ch_1(E_K)$ are coprime.  By \cite[Theorem 4.5]{Lo1}, we can extend $E_K$ to a flat family $E \in D^b(X_R)$ of objects in $\Ap$ over $\Spec R$, so that $j^\ast E \cong E_K$ and $L\iota^\ast E \in \Ap (X_k)$.  Suppose $L\iota^\ast E$ is not PT-semistable.  Then $L\iota^\ast E$ has a maximal destabilising subobject $E_0 \hookrightarrow L\iota^\ast E$ in $\Ap (X_k)$.

On the other hand, $p^\ast E$ is also a flat family over $\hat{R}$, and $L\iota^\ast (p^\ast E) \cong L\iota^\ast E$.  Moreover, $j^\ast (p^\ast E) \in \Ap_{\hat{K}}$ by Lemma \ref{PTsspullbackinheart} below.  Since $j^\ast (p^\ast E) \cong p^\ast (j^\ast E)$, $j^\ast (p^\ast E)$ is also PT-semistable by Proposition \ref{pro-PTssbasechange} below.  Moreover, $L\iota^\ast (p^\ast E)$ also has $E_0$ as a maximal destabilising subobject in $\Ap (X_k)$.

We can now apply elementary modifications to the family $E$.  Since any flat family over $\Spec R$ has the same central fibre as the flat family obtained after the base change via $p^\ast$, by our result over a complete DVR, Theorem \ref{theorem-prelimcompletenessofPTss}, the elementary modifications applied to $E \in D^b(X_R)$ must produce a semistable central fibre after finitely many steps. This proves Theorem \ref{theorem-prelimcompletenessofPTss} again with the additional hypothesis.

We end this section with the proofs of Lemma \ref{PTsspullbackinheart} (the heart is preserved under base change) and Proposition \ref{pro-PTssbasechange} (PT-semistability is preserved under base change).  To this end, we first characterise PT-semistable objects under a coprime assumption on the Chern character:

\begin{pro}\label{pro-PTsscharacterisation}
Let $X$ be a smooth projective three-fold over $k$.  Suppose $E \in \Ap (X)$ has nonzero $ch_0(E)$, and $ch_0(E), ch_1(E)$ are coprime.  Then $E$ is PT-semistable if and only if it is PT-stable, if and only if the following conditions hold:
\begin{itemize}
\item[(1)] $H^{-1}(E)$ is torsion-free and $\mu$-semistable;
\item[(2)] $H^0(E)$ is 0-dimensional;
\item[(3)] $\Hom_{D(X)} (\OO_x, E)=0$ for any $x \in X$.  (Here, $\OO_x$ denotes the skyscraper sheaf with value $k$ supported at $x$.)
\end{itemize}
\end{pro}

\begin{proof}
Suppose $E$ is PT-semistable.  Then properties (1) to (3) follow from \cite[Lemma 3.3]{Lo1} (which partially characterises PT-semistable objects) and the fact that $\phi (\rho_0) > \phi (-\rho_3)$ in the definition of the central charge for PT-stability.

Suppose $E$ satisfies conditions (1) to (3).  Let $E_0 \hookrightarrow E$ be a maximal destabilising subobject.  Then $\phi (E_0) \succ \phi (E)$, and so $E_0$ cannot be 1-dimensional.  By condition (3), $E_0$ is either 2-dimensional or 3-dimensional.  By condition (1), $E_0$ must be 3-dimensional, and $H^{-1}(E_0)$ must be torsion-free with nonzero rank.  Since $\phi (E_0) \succ \phi (E)$, we must have $\mu (H^{-1}(E_0)) \geq \mu (H^{-1} (E))$, forcing $\mu (H^{-1}(E_0)) = \mu (H^{-1}(E))$.  But the degree and rank of $H^{-1}(E)$ are coprime, and so the degrees and ranks of $H^{-1}(E_0)$ and $H^{-1}(E)$ must agree.  So the inclusion $H^{-1}(E_0) \hookrightarrow H^{-1}(E)$ must be an isomorphism (since its cokernel is supposed to be in $\Coh_{\geq 2}(X)$), implying $H^0(E_0) \hookrightarrow H^0(E)$, contradicting $\phi (E_0) \succ \phi (E)$ (since $H^0(E)$ is 0-dimensional by condition (3)).  Therefore, $E$ must have been PT-semistable to start with.

It remains to show that, if $E$ is PT-semistable, then it is PT-stable.  Suppose $F \hookrightarrow E$ is a subobject such that $\phi (F) = \phi (E)$.  Then we must have $\mu (H^{-1}(F))=\mu (H^{-1}(E))$.  As above, the coprime assumption forces $H^{-1}(F) \hookrightarrow H^{-1}(E)$ to be an isomorphism, and so the induced map $H^0(F) \to H^0(E)$ is injective.  By $\phi (F) = \phi (E)$, the inclusion $H^0(F) \hookrightarrow H^0(E)$ must be an isomorphism.  As a result, the inclusion $F \hookrightarrow E$ is a quasi-isomorphism, i.e.\ $F$ and $E$ are isomorphic objects in $D^b(X)$.  This means that there are no strictly PT-semistable objects in $\Ap (X)$, so $E$ is PT-stable.
\end{proof}

We have the easy corollary

\begin{coro}\label{ch0ch1coprimeE00dim}
With the same hypotheses as above, if $H^{-1}(E)$ is torsion-free and $\mu$-semistable (or, equivalently in this case, $\mu$-stable), $H^0(E)$ is 0-dimensional, and $E$ itself is PT-unstable, then any maximal destabilising subobject of $E$ is 0-dimensional.
\end{coro}

\begin{lemma}[Base change preserves the heart]\label{PTsspullbackinheart}
Let $L/K$ be a field extension.  Let $p : X_L \to X_K$ be the induced morphism. If $E \in \Ap (X_K)$, then $p^\ast E \in \Ap (X_L)$.
\end{lemma}

The following proof was suggested to me by Ziyu Zhang.

\begin{proof}
Take any $F \in \Coh_{\geq 2}(X_K)$.  Consider the torsion-filtration of $p^\ast F$,
$$ G_0 \subseteq G_1 \subseteq G_2 \subseteq G_3= p^\ast F$$
where $G_i$ is the  maximal subsheaf of $p^\ast F$ of dimension $\leq i$; this filtration is unique \cite[p.3]{HL}.  Take any $\sigma \in \Aut (L/K)$.  Then for any $i$, we find that $\sigma^\ast G_i$ is again a subsheaf of $p^\ast F$ of dimension at most $i$.  By the maximality of $G_i$, we have $\sigma^\ast G_i \subseteq G_i$.  However, $\sigma^\ast G_i$ and $G_i$ have the same Hilbert polynomial, and so they must be equal.  Thus each $G_i$ is invariant under $\sigma^\ast$ for all $\sigma \in \Aut (L/K)$, and so by descent theory for sheaves, there are subsheaves $\tilde{G}_i$ of $F$ in $\Coh (X_K)$ such that $G_i = \tilde{G}_i \otimes_K L$ for each $i$.  If $G_1 \neq 0$, then $\tilde{G}_1\neq 0$ also, contradicting $F \in \Coh_{\geq 2}(X_K)$.  Hence $p^\ast F \in \Coh_{\geq 2}(X_L)$.

That $p^\ast (\Coh_{\leq 1}(X_K)) \subseteq \Coh_{\leq 1}(X_L)$ is clear.  Thus, since $p^\ast$ is exact, for any $E \in \Ap (X_K)$ we do have $p^\ast E \in \Ap (X_L)$.
\end{proof}

\begin{pro}[Base change preserves PT-semistability]\label{pro-PTssbasechange}
Let $k \subset K \subset L$ be field extensions, and $p : X_L \to X_K$ the induced map, where $X$ is a smooth projective  three-fold over $k$.  Suppose $E \in \Ap (X_K)$, with $ch_0(E)\neq 0$ and $ch_0(E), ch_1(E)$ coprime.  If $E$ is PT-semistable, then $p^\ast E$ is also PT-semistable.
\end{pro}

\paragraph{Notation.}  Following \cite{TodaLSOp}, we define, for a three-dimensional Noetherian scheme $Y$, the following subcategories of $\Ap (Y)$
\begin{align*}
  \Ap_1 (Y) &:= \langle F[1], \OO_x : F \text{ is a pure 2-dimensional sheaf}, x \in X \rangle, \\
 \Ap_{1/2} (Y) &:= \{ E \in \Ap (Y) : \Hom (F,E)=0 \text{ for all }F \in \Ap_1\}.
\end{align*}
As mentioned in \cite{TodaLSOp}, an object $E \in \Ap (Y)$ is in $\Ap_1 (Y)$ iff $H^{-1}(E)$ is 2-dimensional and $H^0(E)$ is 0-dimensional, while it is in $\Ap_{1/2}(Y)$ iff $H^{-1}(E)$ is torsion-free and $\Hom (\OO_x, E)=0$ for any skyscraper sheaf $\OO_x$.  For any smooth projective three-fold $Y$, we also define the dualising functor
\[ \mathbb{D} : D^b(Y) \to D^b(Y)^{op} : E \mapsto R\HHom (E,\OO_Y[2]).\]

\begin{proof}
Since $p^\ast$ preserves the Chern character, using Corollary \ref{ch0ch1coprimeE00dim}, we just need to show that $\Hom (\OO_x, p^\ast E)=0$ for all closed point $x \in X_L$, where $\OO_x$ denotes the skyscraper sheaf with value $L$ supported at $x$.

Since $E$ is PT-semistable, $E \in \Ap_{1/2}$.  Then $\mathbb{D}(E) \in \Ap_{1/2} \subset \Ap$ by \cite[Lemma 2.17]{TodaLSOp}.  And so $p^\ast \mathbb{D}(E) \in \Ap$ by Lemma \ref{PTsspullbackinheart}.  However, $\mathbb{D}(p^\ast E) \cong p^\ast \mathbb{D}(E)$, so $\mathbb{D}(p^\ast E) \in D^{[-1,0]}_{\Coh (X_L)}$, implying $\Hom_{D(X_L)}(\mathbb{D}(p^\ast E), \OO_x [-1])=0$ for all closed point $x \in X_K$.  Since
\begin{align*}
  \Hom_{D(X_L)}(\mathbb{D}(p^\ast E), \OO_x [-1]) &\cong \Hom_{D(X_L)}(\mathbb{D}(p^\ast E), \mathbb{D}(\OO_x)) \\
  &\cong \Hom_{D(X_L)}(\OO_x, p^\ast E),
  \end{align*}
we have shown that $p^\ast E$ cannot be PT-unstable.
\end{proof}

\paragraph{Remark.}  The proof above takes a different track from the sheaf case.  In the sheaf case \cite{Langton}, the argument is as such: suppose $L/K$ is a field extension over $k$.  Form the base extension
\[
\xymatrix{
  X_L \ar[r] \ar[d]^p & \Spec L \ar[d] \\
  X_K \ar[r] & \Spec K
}
\]
where $X_K = X \otimes_k K$ and $X_L = X \otimes_k L$.  Given a $\mu$-semistable sheaf $E \in \Coh (X_K)$, we show that $p^\ast E$ is $\mu$-semistable by considering the maximal destabilising subobject $E_0 \subseteq p^\ast E$.  Since the inclusion $E_0 \hookrightarrow p^\ast E$ is invariant under Galois automorphisms of $L/K$, we can use descent theory for sheaves to conclude that $E_0$  descends to $X_K$, i.e.\ is the pullback of a sheaf on $X_K$.  However, this argument relies on the uniqueness of HN filtrations for sheaf stability conditions.  For polynomial stabilities in the derived category, we only have uniquness of HN filtrations \textit{up to isomorphism}, which is not  enough in order to make use of descent theory for sheaves.

\section{Openness and Separatedness of PT-Semistability}

\subsection{Openness of PT-Semistability}

By Proposition \ref{pro-PTsscharacterisation}, given any PT-semistable object $E\in \Ap$ with nonzero rank, the following conditions are satisfied:
\begin{enumerate}
\item[A.] $H^{-1}(E)$ is torsion-free and semistable in $\Coh_{3,1}(X)$;
\item[B.] $H^0(E)$ is 0-dimensional;
\item[C.] $\Hom_{D(X)}(k_x,E)=0$ for all $x \in X$.
\end{enumerate}
(That $H^{-1}(E)$ is semistable in $\Coh_{3,1}(X)$, and not just $\mu$-semistable, is not hard to see.)

Let us show that these three conditions together form an open condition for flat families of objects in $\Ap$ in the derived category.  To prove this, we use the following fact: in a Zariski topological space, a set is open if and only if it is constructible and stable under generisation.

\begin{pro}\label{openness-valuativeversion}
For flat families of objects in $\Ap$ of nonzero rank, properties A, B and C together form an open condition.
\end{pro}

\begin{proof}
To start with, let us prove that properties A, B and C together are stable under generisation.  Let $(R,\pi)$ be a DVR, and let $E$ be a flat family of objects in $\Ap$ over $\Spec R$ such that the central fibre $L\iota^\ast E$ satisfies propeties A, B and C.

Suppose $H^{-1}(j^\ast E)$ is not torsion-free.  Then we can find a pure 2-dimensional subsheaf $F \subset j^\ast H^{-1}(E)$ with torsion-free cokernel, giving us an injection $F[1] \hookrightarrow j^\ast H^{-1}(E)[1]$ in $\Ap$.  We can extend $F$ to an $R$-flat family $F_R$ of pure 2-dimension sheaves, and extend the composition $F[1] \hookrightarrow j^\ast H^{-1}(E)[1] \hookrightarrow j^\ast E$ to a morphism $\al_R : F_R \to E$ whose derived restriction to $X_k$ is nonzero.  But then $\image (L\iota^\ast \al_R)$, a sheaf sitting at degree $-1$, would be a 2-dimensional subobject of $L\iota^\ast E$, contradicting $H^{-1}(L\iota^\ast E)$ being torsion-free.  Thus $H^{-1} (j^\ast E)$ must be torsion-free.  The proof that $\Hom_{D(X_k)} (k_x,L\iota^\ast E)=0$ for all $x \in X_k$ implies $\Hom_{D(X_K)} (k_x,j^\ast E)=0$ for all $x \in X_K$ is similar.

To see why $H^0(j^\ast E)$ is 0-dimensional, write $E$ as a complex $E^\bullet = [\cdots \to E^{-1} \to E^0]$ where each $E^i$ is a locally free sheaf sitting at degree $i$.  By assumption, $H^0(L\iota^\ast E) \cong \iota^\ast H^0(E^\bullet)$ is 0-dimensional.  By semicontinuity, $H^0(j^\ast E) \cong j^\ast H^0(E)$ is also 0-dimensional.

Next, suppose $H^{-1}(j^\ast E)$ is not semistable in $\Coh_{3,1}$.  Let $F \hookrightarrow j^\ast H^{-1}(E)$ be a torsion-free sheaf that represents the maximal destabilising subsheaf in $\Coh_{3,1}$.  Then we have the inequality $p_{3,1}(F) \succ p_{3,1} (j^\ast H^{-1}(E))$ for the reduced Hilbert polynomials.  We can then extend $F$ to a torsion-free sheaf $F_R$ on $X_R$ whose central fibre is also torsion-free and semistable in $\Coh_{3,1}$ \cite[Theorem 2.B.1]{HL}.  Then the composition $F[1] \hookrightarrow j^\ast H^{-1}(E)[1] \hookrightarrow j^\ast E$ extends to a morphism $F_R[1] \to E$ on $X_R$ whose central fibre is a nonzero morphism $\al_0 : i^\ast F_R[1] \to L\iota^\ast E$.  Note that $\image (\al_0)$ is a sheaf sitting at degree $-1$.

Now we have an exact sequence of coherent sheaves
\[
  i^\ast F_R \overset{\beta}{\to} \image (\al_0) [-1] \to H^0(\kernel \al_0) \to 0,
\]
and $p_{3,1}(i^\ast F_R) \preceq p_{3,1} (\image \beta) \preceq p_{3,1}(\image (\al_0) [-1])$.  On the other hand, we have $p_{3,1}(F) \succ p_{3,1} (j^\ast H^{-1}(E))=p_{3,1}(H^{-1}(L\iota^\ast E))$ (the last equality follows from the $R$-flatness of $E$, and both $H^0(L\iota^\ast E)$ and $H^0 (j^\ast E)$ being 0-dimensional).  Overall, we have $p_{3,1}(\image (\al_0)[-1]) \succ p_{3,1}(H^{-1}(L\iota^\ast E))$, making $\image (\al_0)[-1]$  a destabilising subsheaf of $H^{-1}(L\iota^\ast E)$ in $\Coh_{3,1}$, a contradiction.

It remains to show that, given a noetherian scheme $S$ over $k$, and a flat family $E$ of objects in $\Ap$ over $S$ with fibres $E_s$ satisfying $ch_0 \neq 0$, the locus in $S$ over which $E_s$ satisfies properties A, B and C is constructible.  Using a flattening stratification, we can assume that each of the cohomology sheaves $H^i(E)$ of $E$ is flat over $S$.  Since each of properties A, B and C is an open condition for flat families of sheaves, we are done.
\end{proof}

\begin{lemma}\label{lemma-mdsbounded}
Fix an integer $r>0$.  For any $d, \beta, n$, define the set of injections in $\Ap$
\begin{multline*}
\mathcal{S} := \{ E_0 \hookrightarrow E : E_0 \text{ is a maximal destabilising subobject of $E$ in $\Ap$},  \\
 \text{where $E$ has properties A, B and C and } ch(E)=(-r,-d,\beta,n) \}.
 \end{multline*}
  Then the set
  \[
 \mathcal{S}_{sub} := \{ E_0 : E_0 \hookrightarrow E \text{ is in }\mathcal{S}\}
 \]
is bounded.
\end{lemma}
\begin{proof}
Take an injection $E_0 \hookrightarrow E$ in $\mathcal{S}$.  Then $E_0$ cannot be 1-dimensional, or else it would not be destabilising for $E$, which is 3-dimensional.  And $E_0$ cannot be 2-dimensional, for $H^{-1}(E)$ is torsion-free.  Since $E$ has property $C$, neither is $E_0$ 0-dimensional.  Hence $E_0$ has to be 3-dimensional.

Since $E_0$ is PT-semistable, we know $H^0(E_0)$ is 0-dimensional, as is $H^0(E)$.  Therefore, the semistability of $H^{-1}(E)$ in $\Coh_{3,1}$ implies that we have the equalities $p_{3,1}(H^{-1}(E_0)) = p_{3,1} (H^{-1}(E))= p_{3,1}(E)$.  And because $\phi (E_0) \succ \phi (E)$, there is a lower bound for $ch_3 (E_0)$, say $ch_3(E_0) \geq \al$.  Then $ch_3 (E/E_0)$ has an upper bound, namely $n-\al$.  Therefore, we have $ch_3 (H^{-1}(E/E_0)) = - (ch_3 (E/E_0)- ch_3 (H^0(E/E_0))) \geq \al - n$, i.e.\ there is a lower bound for $ch_3 (H^{-1}(E/E_0))$.

On the other hand, $H^{-1}(E/E_0)$ and $H^{-1}(E)$ have the same $p_{3,1}$, while we have a quotient $H^{-1}(E) \twoheadrightarrow H^{-1}(E/E_0)$ in $\Coh_{3,1}$.  By the semistability of $H^{-1}(E)$ in $\Coh_{3,1}$, $H^{-1}(E/E_0)$ is also semistable in $\Coh_{3,1}$.  In particular, $H^{-1}(E/E_0)$ is $\mu$-semistable.  Therefore, by \cite[Theorem 4.8]{MaruBFTFS}, the set
\[
\{ H^{-1}(E/E_0) : E_0 \hookrightarrow E \text{ is in }\mathcal{S} \}
\]
is bounded.  Therefore, $\{ ch_3 (H^{-1}(E/E_0)) : E_0 \hookrightarrow E \text{ in }\mathcal{S} \}$ is bounded, which implies $\{ ch_3 (H^0 (E/E_0)) : E_0 \hookrightarrow E \text{ in } \mathcal{S} \}$ is bounded from above (since $ch_3 (E/E_0)$ is so, from the previous paragraph), hence bounded.  However, for any $E_0 \hookrightarrow E$ in $\mathcal{S}$, the sheaf $H^0(E/E_0)$ is a quotient of the 0-dimensional sheaf $H^0(E)$.  Hence $\{ H^0 (E/E_0) : E_0 \hookrightarrow E \text{ in } \mathcal{S} \}$ is bounded, implying \[
\{E/E_0 : E_0 \hookrightarrow E \text{ in } \mathcal{S}\}
\]
 is bounded.  This is a little more than we need.  However, that $\{ ch (E/E_0) : E_0 \hookrightarrow E \text{ in }\mathcal{S}\}$ is bounded now implies $\{ ch (E_0) : E_0 \hookrightarrow E \text{ in }\mathcal{S}\}$ is bounded.  This, coupled with the  boundedness of PT-semistable objects \cite[Proposition 3.4]{Lo1}, implies that the set
 \[
 \mathcal{S}_{sub} = \{ E_0 : E_0 \hookrightarrow E \text{ in } \mathcal{S} \}
 \]
 itself is bounded.
\end{proof}

\begin{pro}[Openness of PT-semistability]\label{pro-openness}
Let $S$ be a Noetherian scheme over $k$, and $E \in D^b(X \times_{\Spec k} S)$ be a flat family of objects in $\Ap$ over $S$ with $ch_0 \neq 0$.  Suppose $s_0 \in S$ is a point such that $E_{s_0}$ is PT-semistable.  Then there is an open set $U \subseteq S$ containing $s_0$ such that for all points $s \in U$, the fibre $E_s$ is PT-semistable.
\end{pro}

\begin{proof}
Since Chern classes are locally constant for flat families of complexes, we may assume that $ch_0(E_s) \neq 0$ for all $s \in S$.   By Proposition \ref{openness-valuativeversion}, we can assume that the fibre $E_s$ satisfies properties A, B and C for every $s \in S$.  Since $\mathcal{S}_{sub}$ is bounded, using the same argument as in \cite[Proposition 3.16]{TodaLSOp} and \cite[Proposition 3.17]{TodaK3}, we know that there exists a scheme $\pi : Q \to S$ of finite type over $S$, a relatively perfect  complex $\mathscr{E}_0 \in D(X \times Q)$ (see \cite[Definition 2.1.1]{Lieblich}) that is a flat family of objects in $\Ap (X)$ over $Q$, and a morphism $\al : \mathscr{E}_0 \to E_Q$ such that, for every $q \in Q$, the fibre $\al_q : (\mathscr{E}_0)_q \to (E_Q)_q=E_{\pi(q)}$ is an injection in $\Ap (X)$ with $\phi (\mathscr{E}_0)_q \succ \phi (E_Q)_q$, and any maximal destabilising subobject $E_0 \hookrightarrow E_s$ for some $s \in S$ occurs as a fibre $\al_q$ of $\al$ where $q \in \pi^{-1} (s)$.

If we can show that $s_0 \notin \overline{\pi (Q)}$, then we can conclude that PT-semistability is an open condition.  Suppose $s_0 \in \overline{\pi (Q)}$.  Then we can find a smooth curve $C$, a closed point $p \in C$, and a map $\gamma : C \to S$ taking $p$ to $s_0$, making the diagram commute:
\[
\xymatrix{
  W \arinj[r] & C \setminus \{p\} \ar[r] \ar[d] & Q \ar[d]^\pi \\
  & C \ar[r]^\gamma & S
}
\]
 where $W$ denotes the generic point of $\Spec \hat{\OO}_{C, p}$.  Let $\overline{W}$ denote $\Spec \hat{\OO}_{C,p}$.  The restriction $\al |_W : \mathscr{E}_0 |_W \to E_{Q} |_W$ is then an injection in $\Ap (X \times W)$ with $\phi ( \mathscr{E}_0 |_W) \succ \phi (E_Q |_W)$.  Let $B_W$ denote a maximal destabilising subobject of $\mathscr{E}_0 |_W$ with respect to PT-semistability.  Since $H^{-1}(E_Q |_W)$ is torsion-free, either $B_W$ is 0-dimensional (hence a sheaf) or is 3-dimensional.  Then we have an injection $\al'_W : B_W \hookrightarrow E_Q |_W$ where $B_W$ is PT-semistable.  By the properness of quot scheme (if $B_W$ is 0-dimensional) or Theorem \ref{theorem-prelimcompletenessofPTss} (if $B_W$ is 3-dimensional, hence of nonzero rank), we can extend $B_W$ to a flat family $B$ of PT-semistable objects over $\overline{W}$.  By \cite[Lemma 3.18]{TodaLSOp}, we can extend $\al'_W$ to a morphism $\al' : B \to E_Q |_{\overline{W}}$ with nonzero central fibre, which gives a destabilising subobject of the central fibre of $E_Q |_{\overline{W}}$, namely $E_{s_0}$, contradicting the PT-semistability of $E_{s_0}$.  This shows $s_0 \notin \overline{\pi (Q)}$, thus proving the proposition.
\end{proof}

The property of being PT-stable is easily seen to be an open property, too:

\begin{pro}[Openness of PT-stability]\label{pro-opennessofPTstab}
Let $S$ be a Noetherian scheme over $k$, and $E \in D^b(X \times_{\Spec k} S)$ be a flat family of objects in $\Ap$ over $S$ with $ch_0\neq 0$.  Suppose $s_0 \in S$ is a point such that $E_{s_0}$ is PT-stable.  Then there is an open set $U \subseteq S$ containing $s_0$ such that for all points $s \in U$, the fibre $E_s$ is PT-stable.
\end{pro}

\begin{proof}
By Proposition \ref{pro-openness}, it suffices to assume that $E$ is a family where each fibre $E_s$ is a PT-semistable object.  Then we can consider the set
\begin{equation*}
\mathcal{S}' := \{ E_0 : E_0 \text{ is a subobject of } E_s \text{ for some }s \in S,  \text{ and }\phi (E_0) = \phi (E_s) \}
\end{equation*}
as in the proof of Lemma \ref{lemma-mdsbounded}.  Since $\{ ch(E_0) : E_0 \in \mathcal{S}'\}$ is bounded, and each $E_0 \in \mathcal{S}'$ is necessarily PT-semistable, the set $\mathcal{S}'$ itself is bounded.  Exactly the same proof as in Proposition \ref{pro-openness} would then yield the result.
\end{proof}

\subsection{Separatedness for PT-Stable Objects}

Note that, for objects $F, G \in D^b(X_R)$, the $R$-module $\Hom_{D^b(X_R)} (F,G)$ is finitely generated and
\begin{equation}\label{HomX_RX_K}
   \Hom_{D^b(X_R)}(F,G) \otimes_R K \cong \Hom_{D^b(X_K)} (F \otimes_R K, G\otimes_R K)
\end{equation}
as in \cite[Lemma 3.18]{TodaLSOp}.  Since $\Coh (Y)$ is a full subcategory of $D^b(Y)$ for $Y=X_K$ or $X_R$ (being the heart of the standard t-structure), when $F,G \in \Coh (X_R)$, the above isomorphism becomes
\begin{equation*}
   \Hom_{\Coh(X_R)}(F,G) \otimes_R K \cong \Hom_{\Coh(X_K)} (F \otimes_R K, G \otimes_R K).
\end{equation*}

\begin{pro}[Valuative criterion for separatedness]\label{pro-separatedness}
Let $X$ be a smooth projective three-fold over $k$, and $R$ an arbitrary DVR over $k$.
Let $E_1, E_2 \in D^b(X_R)$ be flat families of objects in $\Ap$ over $\Spec R$, and suppose $j^\ast E_1 \cong j^\ast E_2$ in $D^b(X_K)$.  Suppose both $L\iota^\ast E_1, L\iota^\ast E_2$ are PT-semistable objects in $\Ap (X_k)$ and at least one of them is PT-stable.  Then the isomorphism $j^\ast E_1 \cong j^\ast E_2$ extends to an isomorphism $E_1 \cong E_2$.
\end{pro}

\begin{proof}
Let $f : j^\ast E_1 \to j^\ast E_2$ be the given isomorphism in $D^b(X_K)$.  Then there exists an integer $m$ such that $\pi^mf$ extends to a morphism $\pi^mf : E_1 \to E_2$ on $X_R$, and $L\iota^\ast (\pi^mf) \neq 0$ by \cite[Lemma 3.18]{TodaLSOp}.  That is, $G:=\image (L\iota^\ast (\pi^mf)) \neq 0$.

If $L\iota^\ast E_1$ is stable, and $\kernel (L\iota^\ast (\pi^mf)) \neq 0$, then $\phi (L\iota^\ast E_1) \prec \phi (G)$ but $\phi (G) \preceq \phi (L\iota^\ast E_2)$, which is impossible since $\phi (L\iota^\ast E_1)=\phi (L\iota^\ast E_2)$ (since $L\iota^\ast E_1, L\iota^\ast E_2$ have the same Chern character).  So $\kernel (L\iota^\ast (\pi^mf))=0$, meaning $L\iota^\ast (\pi^mf)$ is an injection between objects in $\Ap$ of the same Chern character on $X_k$, and is necessarily an isomorphism.

If $L\iota^\ast E_2$ is stable, and $\cokernel (L\iota^\ast (\pi^mf)) \neq 0$, then  $\phi (L\iota^\ast E_1) \preceq \phi (G)$ while $\phi (G) \prec \phi (L\iota^\ast E_2)$, again a contradiction.  So $\cokernel (L\iota^\ast (\pi^mf))$ must be 0, i.e.\ $L\iota^\ast (\pi^mf)$ is a surjection between objects in $\Ap$ of the same Chern character, and is necessarily an isomorphism.

In any case, $\pi^m f : E_1 \to E_2$ restricts to an isomorphism on $X_k$.  Let $M$ be the homotopy kernel of $\pi^m f$, so we have an exact triangle in $D^b(X_R)$
\[
  M \to E_1 \overset{\pi^m f}{\to} E_2 \to M[1].
\]
That $L\iota^\ast (\pi^m f)$ is an isomorphism means $L\iota^\ast M = 0$.  And this implies $M$ itself is zero by \cite[Lemma 2.1.4]{Lieblich}.  That is, $\pi^m f$ is an isomorphism.
\end{proof}

\section{Moduli Spaces of PT-Semistable Objects}\label{section-modulidef}

%Explain the statement of Theorem \ref{theorem-main}.

By Lieblich's result \cite[Corollary A.4]{BSMSKTS}, there is an Artin stack locally of finite presentation
\[
\mathscr{D}_{\Ap} (X) \to \Spec k
\]
such that, for any scheme $B$ over $\Spec k$, the fibre of the stack over $B$ is the category of complexes $E \in D^b(X_B)$ such that for each  point $b \hookrightarrow B$, the fibre $E|_b$ obtained by derived pullback lies in the heart $\Ap (X_b)$.  If we fix a Chern character $ch=(-r,-d,\beta,n)$, then we have a closed substack $\mathscr{D}_{\Ap}^{ch} (X)$ whose fibres are categories of families of complexes in $\Ap$ with the prescribed Chern character.  Then, by openness of PT-semistability (Proposition \ref{pro-openness}) and openness of PT-stability (Proposition \ref{pro-opennessofPTstab}), we have open substacks
\begin{equation*}
\mathscr{D}_{\Ap}^{ch, PTs} (X) \subset \mathscr{D}_{\Ap}^{ch, PTss} (X)  \subset  \mathscr{D}_{\Ap}^{ch} (X)
\end{equation*}
whose fibres are categories of families of PT-stable and PT-semistable objects, respectively, with Chern character $ch$.

By boundedness of PT-semistable objects \cite[Proposition 3.4]{Lo1}, both the moduli stacks $\mathscr{D}_{\Ap}^{ch, PTs}(X)$ and $\mathscr{D}_{\Ap}^{ch, PTss} (X)$ are Artin stacks of finite type.  Theorem \ref{theorem-prelimcompletenessofPTss} says that the stack $\mathscr{D}_{\Ap}^{ch, PTss} (X)$ is universally closed, while $\mathscr{D}_{\Ap}^{ch, PTs}(X)$ is separated by Proposition \ref{pro-separatedness}.

 When there are no strictly semistable objects, such as when $r,d$ are coprime, we have that $\mathscr{D}_{\Ap}^{ch, PTss} (X)=\mathscr{D}_{\Ap}^{ch, PTs} (X)$ is a proper Artin stack of finite type.  Alternatively,  as in \cite{TodaLSOp}, we can consider the moduli functor from schemes over $\Spec k$ to sets
\[
\MM : (\text{Sch}/\Spec k) \to (\text{Sets})
\]
that takes a scheme $B$ over $\Spec k$ to the set of complexes $E \in D^b(X_B)$ such that, for each  point $b \in B$, the fibre $E|_b$ obtained by derived pullback satisfies $\Hom (E|_b,E|_b)=k$ and $\Ext^{-1}(E|_b, E|_b) = 0$.  Inaba \cite[Theorem 0.2]{Inaba} showed that the \'{e}tale sheafification of this functor is represented by a locally separated algebraic space, in the sense of \cite{Knutson}.  Fixing a Chern character $ch$, we obtain a subfunctor $\MM^{\text{\'{e}t}}_{ch,PTs}$ that takes a scheme $B$ to the set of families of PT-stable objects over $B$ with Chern character $ch$.  As in the previous paragraph, by our results on boundedness, openness, separatedness (Proposition \ref{pro-separatedness}) and universal closedness, $\MM^{\text{\'{e}t}}_{ch,PTs}$ is a proper algebraic space of finite type.

This completes the proof of Theorem \ref{theorem-main}.

\end{document}